\documentclass{amsart}

\usepackage{psfrag}
\usepackage{color}
\usepackage{tikz}
\usetikzlibrary{matrix,arrows}
\usepackage{graphicx,graphics}
\usepackage{amssymb,amsfonts,amsmath,amstext,amsthm,amscd,verbatim,enumerate,soul}
\usepackage[T1]{fontenc}

\begin{document}

\newtheorem{theorem}{Theorem}[section]
\newtheorem{result}[theorem]{Result}
\newtheorem{fact}[theorem]{Fact}
\newtheorem{conjecture}[theorem]{Conjecture}
\newtheorem{lemma}[theorem]{Lemma}
\newtheorem{proposition}[theorem]{Proposition}
\newtheorem{corollary}[theorem]{Corollary}
\newtheorem{facts}[theorem]{Facts}
\newtheorem{props}[theorem]{Properties}
\newtheorem*{thmA}{Theorem A}
\newtheorem{ex}[theorem]{Example}
\theoremstyle{definition}
\newtheorem{definition}[theorem]{Definition}
\newtheorem*{remark}{Remark}
\newtheorem{example}[theorem]{Example}
\newtheorem*{defna}{Definition}

\newcommand{\notes} {\noindent \textbf{Notes.  }}
\newcommand{\defn} {\noindent \textbf{Definition.  }}
\newcommand{\defns} {\noindent \textbf{Definitions.  }}
\newcommand{\x}{{\bf x}}
\newcommand{\e}{\epsilon}
\renewcommand{\d}{\delta}
\newcommand{\z}{{\bf z}}
\newcommand{\B}{{\bf b}}
\newcommand{\V}{{\bf v}}
\newcommand{\T}{\mathbb{T}}
\newcommand{\Z}{\mathbb{Z}}
\newcommand{\Hp}{\mathbb{H}}
\newcommand{\D}{\Delta}
\newcommand{\R}{\mathbb{R}}
\newcommand{\N}{\mathbb{N}}
\renewcommand{\B}{\mathbb{B}}
\renewcommand{\S}{\mathbb{S}}
\newcommand{\C}{\mathbb{C}}
\newcommand{\rt}{\widetilde{\rho}}
 \newcommand{\adj}{{\mathrm{adj}\;}}
 \newcommand{\0}{{\bf O}}
 \newcommand{\av}{\arrowvert}
 \newcommand{\zbar}{\overline{z}}
 \newcommand{\xbar}{\overline{X}}
 \newcommand{\htt}{\widetilde{h}}
\newcommand{\ty}{\mathcal{T}}
\newcommand\diam{\operatorname{diam}}
\renewcommand\Re{\operatorname{Re}}
\renewcommand\Im{\operatorname{Im}}
\newcommand{\tr}{\operatorname{Tr}}
\renewcommand{\skew}{\operatorname{skew}}
\newcommand{\vol}{\operatorname{vol}}
\newcommand{\dt}{\widetilde{\mathcal{D}}}
\newcommand{\ft}{\widetilde{f}}
\newcommand{\gD}{\mathcal{D}}
\newcommand{\ct}{\widetilde{C}}

\newcommand{\ds}{\displaystyle}
\numberwithin{equation}{section}
%
%
%
%
\newcommand{\M}{{\mathcal{M}}}
\newcommand{\tef}{transcendental entire function}
\newcommand{\qfor}{\quad\text{for }}
\newcommand*{\defeq}{\mathrel{\vcenter{\baselineskip0.5ex \lineskiplimit0pt
 \hbox{\scriptsize.}\hbox{\scriptsize.}}}
 =}
%
%

\renewcommand{\theenumi}{(\roman{enumi})}
\renewcommand{\labelenumi}{\theenumi}

\newcommand{\alastair}[1]{{\scriptsize \color{blue}\textbf{Alastair's note:} #1 \color{black}\normalsize}}

\title{A unified approach to Quasiregular Linearization in the plane}

\author{Alastair N. Fletcher}
\address{Department of Mathematical Sciences, Northern Illinois University, DeKalb, IL}
\email{afletcher@niu.edu}

\author{Jacob Pratscher}
\address{Department of Mathematics and Statistics, Stephen F Austin University, 1936 North Street, Nacogdoches, TX}
\email{jacob.pratscher@sfasu.edu}

\date{\today}

\begin{abstract}
We generalize the classical K\"onig's and B\"ottcher's Theorems in complex dynamics to certain quasiregular mappings in the plane. Our approach to these results is unified in the sense that it does not depend on the local injectivity, or not, of the map at the fixed point. By passing to the logarithmic transform we obtain a quasiconformal mapping in either case. Certain restrictions on the quasiregular mappings are needed in order for there to be a candidate to linearize to. These are provided by requiring a simple infinitesimal space of the mapping at the fixed point and restricting to the BIP mappings introduced by the authors in prior work \cite{FP}.
\end{abstract}

\maketitle

\section{Introduction}
\label{sec:intro}

\subsection{Background}

Linearization is an important concept in dynamics. By conjugating a map near a fixed point to a simpler canonical map, information about the original map may be revealed. The motivation for the current paper comes from the classical K\"onig's and B\"ottcher's Theorems in complex dynamics, see for example \cite{CG,Mil}. These fundamental results date from the nineteenth and early twentieth centuries and essentially say that, in most cases, near a fixed point the behaviour of a holomorphic function can be read off from the first non-constant term in the Taylor series.

More precisely, suppose that $f$ is holomorphic in a neighbourhood of $z_0=0$ and that $f(0)= 0 $. K\"onig's Theorem states that if $|\lambda |$ is not $0$ nor $1$, then 
\[ f(z) = \lambda z + \sum_{n=2}^{\infty} a_n z^n \]
may be holomorphically conjugated to the map $z \mapsto \lambda z$ in a neighbourhood of $0$. On the other hand, B\"ottcher's Theorem deals with the case when $\lambda  = 0$ and states that if 
\[ f(z) = \sum_{n=d}^{\infty} a_n z^n,\]
for some $d\geq 2$, then $f$ may be conjugated to $z\mapsto z^d$ in a neighbourhood of $0$. We will not touch upon the rich possibilities that can occur when $|\lambda | = 1$.

Our focus in this paper will be on the question of linearizing quasiconformal and quasiregular maps in a neighbourhood of a fixed point, say $z_0=0$, in the plane. Quasiregular maps are a natural generalization of holomorphic functions and are now a standard tool in complex dynamics, see the book of Branner and Fagella \cite{BF}. The iteration of quasiregular maps themselves have recently received more attention, see for example the survey of Bergweiler \cite{B} for an entry point to this theory.

An immediate obstacle to progress in the linearization problem for quasiregular maps arises in the possible non-differentiability at a fixed point. There is no general conformal model to conjugate to in this setting. The resolution to this in our approach is to ensure there is a natural candidate to conjugate our map to by hypothesizing the issue away.

To obtain this candidate, we use generalized derivatives as introduced by Gutlyanskii et al in \cite{GMRV}. A generalized derivative of a quasiregular map $f:U\to \R^n$ at $x_0 \in U$ is defined to be any local uniform limit of 
\[ \frac{ f(x_0 + r_k x) - f(x_0) }{ \rho_f(r_k) } \]
as $r_k \to 0$, where $\rho_f(r_k)$ is the mean radius of the image of $B(x_0,r_k)$ under $f$.
Of course, not every such sequence need have a limit, but the quasiregular version of Montel's Theorem implies there will be a subsequence along which there is local uniform convergence. See \cite[p.103]{GMRV} for a discussion of this point. The collection of generalized derivatives of $f$ at $x_0$ is called the infinitesimal space $T(x_0,f)$. It was shown in \cite[Corollary 2.11]{FW} that $T(x_0,f)$ contains either one element or uncountably many.

In the former case where $T(x_0,f)$ consists of only one map $g$, then $f$ is called simple at $x_0$. For example, if $f$ is differentiable at $x_0$, then $f$ is simple at $x_0$. If $f$ is simple at $x_0$, then $f$ has an asymptotic representation analogous to a first degree Taylor polynomial approximation of an analytic function. As is customary, for simplicity we assume throughout the paper that the fixed point is at $x_0=0$. Then \cite[Proposition 4.7]{GMRV} states that as $x\to 0$, we have
\[ f(x) \sim \mathcal{D} (x) := \rho_f(|x|) g(x/ |x| ),\]
where $p(x) \sim q(x)$ as $x\to 0$ means
\[ |p(x) - q(x) | = o( |p(x)| + |q(x) | ).\]
The map $\mathcal{D}$ is called the asymptotic representative of $f$ at (in this case) $0$. Informally speaking, the $g(x/|x|)$ term takes care of the shape of $f$ near $0$ and the $\rho_f(|x|)$ term takes care of the scaling.

The mean radius function $\rho_f$ may be badly behaved enough that the asymptotic representative is not quasiconformal, see \cite[Proposition 1.4]{FP}. However, a new sub-class of quasiregular maps called bounded integrable parameterization maps (or BIP maps for brevity) were introduced by the current authors in \cite{FP}. Under the assumption that the map is locally injective at a fixed point, the BIP condition guarantees the quasiconformality of the asymptotic representative, at least in a neighbourhood of the fixed point, see \cite[Theorem 1.5]{FP}. The definition of BIP maps is somewhat technical, and we defer recollection of them until the next section.

In this paper, we extend the BIP condition for maps which are not locally injective at a fixed point. However, it turns out that the logarithmic transform of the asymptotic representative is quasiconformal in either case.
One of the main goals of this paper is to give a unified approach to linearization results where the local injectivity, or not, of the map at the fixed point is not the point, but the attractive or repulsive behaviour is. For holomorphic maps, non-injectivity at a fixed point means the fixed point is superattracting, but in the quasiregular setting this need not be the case. 

\subsection{Statement of results}

To build towards the statement of our main result, we recall that a quasiregular map in a plane domain is locally quasiconformal away from its branch set, which is necessarily discrete in dimension two. Every quasiregular map $f:U \to \C$ therefore has an associated complex dilatation $\mu_f \in L^{\infty}(U)$ which satisfies $||\mu_f||_{\infty} \leq k <1$ for some $k\in [0,1)$. 

\begin{definition}
\label{def:fixpt}
Let $f$ be quasiregular in a neighbourhood of $z_0 \in \C$ with $f(z_0) = z_0$. 
\begin{enumerate}[(i)]
\item We say that $z_0$ is a geometrically attracting fixed point if there exist $\lambda \in (0,1)$ and a neighbourhood $U$ of $z_0$ such that
\begin{equation}
\label{eq:attrfixpt}
|f(z) - z_0| \leq \lambda |z-z_0|
\end{equation}
for $z\in U$. 
\item If for every $\lambda \in (0,1)$ there exists a neighbourhood $U_{\lambda}$ such that \eqref{eq:attrfixpt} holds for $z\in U_{\lambda}$, then we say that $z_0$ is a superattracting fixed point.
\item We say that $z_0$ is a geometrically repelling fixed point if there exist $\lambda >1$ and a neighbourhood $U$ of $z_0$ such that
\begin{equation}
\label{eq:repfixpt}
|f(z) - z_0| \geq \lambda |z-z_0|
\end{equation}
for $z\in U$. 
\item If for every $\lambda >1$ there exists a neighbourhood $U_{\lambda}$ such that \eqref{eq:repfixpt} holds for $z\in U_{\lambda}$, then we say that $z_0$ is a superrepelling fixed point.
\end{enumerate}
\end{definition}

As far as the authors are aware, the notion of a superrepelling fixed point has not appeared in the literature before.

\begin{example}
Examples of all four cases in Definition \ref{def:fixpt} are provided by considering maps of the form $f(z) = C z^n |z|^m$ for $n \in \N$ and $n+m > 0$. The map $f(z) = z|z|$ has a superattracting fixed point at $0$, whereas the map $f(z) = z^2|z|^{-3/2}$ has a superrepelling fixed point at $0$. Note also that this latter map is not injective at $0$.
\end{example}

Suppose that $f$ is a quasiregular map which fixes $0$ and that the local index $i(0,f) = d>1$. As $f$ is an open and discrete map, there exists $r>0$ so that $f |_{ B(0,r) \setminus \{ 0 \} }$ is a covering map. The image is also topologically a punctured disk. As $\exp : \{ z : \Re(z) < \log r \} \to B(0,r) \setminus  \{ 0 \}$ is a universal covering map, it follows that there exists a lift $\ft : \{ z : \Re(z) < \log r \} \to \C$ of $f$ which satisfies $\exp \circ \ft = f \circ \exp$. While we could take $\ft$ to be $2\pi i $ periodic, for our purposes we wish $\ft$ to be injective on its domain. We therefore require
\[ \ft(z+2\pi i ) = \ft(z) + 2\pi d i \]
for $\Re(z) < \log r $. Locally, we may write
\[ \ft(z) = \log f(e^z)\]
for some branch of the logarithm.

\begin{example}
If $f(z) = Cz^n |z|^m$, then its logarithmic transform is
\[ \ft(z) = (n+m)x + niy + \log C,\]
that is, $\ft$ is an affine map. For this map, $i(0,f) = n$ and we see that
\[ \ft(z+2\pi i) = (n+m)x + (ny + 2\pi n) i + \log C = \ft(z) +2\pi n i.\]
Two important examples to bear in mind are as follows. If $f(z) = z/2$ then $\ft(z) = z - \log 2$ and if $f(z) = z^2$ then $\ft(z) = 2z$.
\end{example}

On the other hand, if an injective map $\widetilde{g}$ defined for $\Re(z) <\log r$ satisfies $\widetilde{g}(z + 2\pi i) = \widetilde{g}(z) + 2\pi d i$, then there is a map $g$ defined in $B(0,r)$ whose logarithmic transform is $\widetilde{g}$.
We will use the logarithmic transform to state the conditions we require from our maps. As can be seen from the results in \cite{FP}, it can be more natural to consider properties in this setting. To reiterate, the point is that in dimension two, the branch set is discrete and we can guarantee that the logarithmic transform is an injective map.

\begin{theorem}[Linearization for attracting fixed points]
\label{thm:1}
Let $f:U \to \C$ be a BIP quasiregular map which fixes $0$ and suppose that $f$ is simple at $0$ with asymptotic representative $\mathcal{D}$ whose logarithmic transform $\dt$ is $L$-bi-Lipschitz. Suppose that there exist positive constants $R,\alpha, \beta, \beta', T_1,T_2,T_3$ so that in $\{z : \Re(z) < \log R \}$ we have:
\begin{enumerate}[(a)]
\item a uniform estimate on how close $\ft$ and $\dt$ are, that is,
\begin{equation}
\label{eq:thm1eq1} 
| \ft(z) - \dt(z)| < T_1e^{\alpha \Re z}, 
\end{equation}
\item the complex derivatives of $\widetilde{\mathcal{D}}$ are locally uniformly H\"older continuous, that is, 
\begin{equation}
\label{eq:thm1eq2} 
| (\dt)_z(u) - (\dt)_z(v)| \leq T_2 |u-v|^{\beta}, \quad  | (\dt)_{\zbar}(u) - (\dt)_{\zbar}(v)| \leq T_2 |u-v|^{\beta},
\end{equation}
for $|u-v| < r$,
\item a uniform estimate on how close the complex derivatives of $\ft$ and $\dt$ are, that is, 
\begin{equation} 
\label{eq:thm1eq3}
| (\ft)_z(u) - (\dt)_z(u) | \leq T_3e^{\beta'\Re(u) }, \quad | (\ft)_{\zbar}(u) - (\dt)_{\zbar}(u) | \leq T_3e^{\beta'\Re(u) }.
\end{equation}
\end{enumerate}
Set $\nu = \min \{\alpha\beta, \beta ' \}$. First, if $f$ has a geometrically attracting fixed point at $0$ with parameter 
\begin{equation}
\label{eq:thm1eq4}
\lambda < \min \{ L^{-1/\alpha} , K(\dt)^{-1/\nu} \},
\end{equation} 
then there exists an asymptotically conformal quasiconformal map $\psi : B(0,R) \to \C$ such that
\[ \psi \circ f = \mathcal{D} \circ \psi \]
in $B(0,R)$. Second, if $f$ has a superattracting fixed point at $0$, then the conclusion holds without any further restrictions such as \eqref{eq:thm1eq4}.
\end{theorem}

There are a lot of assumptions made in Theorem \ref{thm:1}, so it is worth justifying them. By \cite[Theorem 1.5]{FP}, the assumptions that $f$ is BIP and simple at $0$ guarantee that $\dt$ is bi-Lipschitz and, in conjunction with \eqref{eq:thm1eq3}, shows that the amount of distortion $f$ has near $0$ carries over to $\dt$. The worse that $\dt$ is behaved, that is, the larger that $L$ needs to be, the stronger restrictions we need on the parameter $\lambda$ for the geometrically attracting fixed point. Moreover, the asymptotic relationship between $f$ and $\mathcal{D}$ only guarantees that $| f(z) - \mathcal{D}(z) | = o(1)$ as $z\to 0$. We need the stronger condition \eqref{eq:thm1eq1} to make our argument work. 

Just with \eqref{eq:thm1eq1}, we can show that there is a topological conjugacy between $f$ and $\mathcal{D}$. However, to promote this to a quasiconformal conjugacy (via a map that is in fact asymptotically conformal) our argument requires \eqref{eq:thm1eq2} and \eqref{eq:thm1eq3}. It is conceivable that these assumptions could be weakened, however, without a continuity condition on $\mu_{\dt}$ the map $\psi$ certainly cannot be guaranteed to be asymptotically conformal.

In the classical K\"onig's Theorem, the logarithmic transform of the map $z\mapsto \lambda z$ is $z\mapsto z + \log \lambda$ which is an isometry and thus $L=1$. Moreover, the Taylor series expansion of $f$ guarantees that we may take $\alpha =1$ as if
\[ f(z) = \lambda z + \sum_{n=2}^{\infty} a_nz^n, \quad A(z) = \lambda z,\]
then
\begin{align*}
| \ft(z) - \widetilde{A}(z)| &= \left | \log ( \lambda e^z + a_2e^{2z} + \ldots ) - \log ( \lambda e^z) \right | \\
&= O(e^{\Re z})
\end{align*}
as $\Re( z)\to -\infty$.
This shows that in this case the condition $L |\lambda|^{\alpha } <1$ reduces to $|\lambda| <1$. It is worth pointing out that if we apply Theorem \ref{thm:1} when $f$ is holomorphic, we don't quite recover the classical K\"onig's Theorem as the mean radius function $\rho_f$ is not linear if $f$ is not. 
We do however recover it from Theorem \ref{thm:1} via the following proposition.

\begin{proposition}
\label{prop:holo}
Suppose $f$ is holomorphic at $z_0=0$, fixes $0$ and satisfies $0<|f'(z_0)| <1$. Then the asymptotic representative $\mathcal{D}$ is smoothly conjugate to the map $z\mapsto f'(0) z$.
\end{proposition}

In a similar way, we also recover B\"ottcher's Theorem.

\begin{proposition}
\label{prop:holo2}
Suppose $f$ is non-constant holomorphic at $z_0=0$, fixes $0$ and satisfies $i(0,f) =d$. Then the asymptotic representative $\mathcal{D}$ is smoothly conjugate to the map $z\mapsto  z^d$.
\end{proposition}

As is to be expected, we also have a version of Theorem \ref{thm:1} for repelling fixed points. However, unlike the classical K\"onigs' Theorem, it is not a trivial task to just apply Theorem \ref{thm:1} to an inverse as $f$ need not be invertible, although $\ft$ is. Then, on a technical level, there is work to do to check that the assumptions in Theorem \ref{thm:1} are preserved under taking an inverse.

\begin{theorem}[Linearization for repelling fixed points]
\label{thm:2}
Let $f:U \to \C$ be a BIP quasiregular map which fixes $0$ and suppose that $f$ is simple at $0$ with asymptotic representative $\mathcal{D}$ whose logarithmic transform $\dt$ is $L$-bi-Lipschitz. Suppose that there exist positive constants $R,\alpha, \beta, \beta', S_1,S_2,S_3$ so that in $\{z : \Re(z) < \log R \}$ we have:
\begin{enumerate}[(a)]
\item a uniform estimate on how close $\ft$ and $\dt$ are, that is,
\begin{equation}
\label{eq:thm2eq1} 
| \ft(z) - \dt(z)| < S_1e^{\alpha \Re z}, 
\end{equation}
\item the complex derivatives of $\widetilde{\mathcal{D}}$ are locally uniformly H\"older continuous, that is, 
\begin{equation}
\label{eq:thm2eq2} 
| (\dt)_z(u) - (\dt)_z(v)| \leq S_2 |u-v|^{\beta}, \quad  | (\dt)_{\zbar}(u) - (\dt)_{\zbar}(v)| \leq S_2 |u-v|^{\beta},
\end{equation}
for $|u-v| < r$,
\item a uniform estimate on how close the complex derivatives of $\ft$ and $\dt$ are, that is, 
\begin{equation} 
\label{eq:thm2eq3}
| (\ft)_z(u) - (\dt)_z(u) | \leq S_3e^{\beta'\Re(u) }, \quad | (\ft)_{\zbar}(u) - (\dt)_{\zbar}(u) | \leq S_3e^{\beta'\Re(u) }.
\end{equation}
\end{enumerate}
Set $\nu = \min \{\alpha\beta, \beta ' \}$. First, if $f$ has a geometrically repelling fixed point at $0$ with parameter 
\begin{equation}
\label{eq:thm2eq4}
\lambda > \max \{ L^{1/\alpha} , K(\dt)^{1/\nu} \},
\end{equation} 
then there exists an asymptotically conformal quasiconformal map $\psi : B(0,R) \to \C$ such that
\[ \psi \circ f = \mathcal{D} \circ \psi \]
in $B(0,R)$. Second, if $f$ has a superrepelling fixed point at $0$, then the conclusion holds without any further restrictions such as \eqref{eq:thm2eq4}.
\end{theorem}

The question of uniqueness of the conjugacy in Theorem \ref{thm:1} and Theorem \ref{thm:2} is a natural one. It is not hard to see that if $\psi$ and $\varphi$ both conjugate $f$ to $\mathcal{D}$, then
\[ (\varphi \circ \psi^{-1} ) \circ \mathcal{D} = \mathcal{D} \circ( \varphi \circ \psi^{-1} ).\]
The question thus reduces to finding a classification of which maps conjugate $\mathcal{D}$ to itself or, equivalently, which maps conjugate $\dt$ to itself. We make some remarks in this direction in the final section.

\subsection{Relationship to previous work}

We now turn to discussing how these results compare with prior work. In \cite{Jia}, Jiang proved versions of K\"onig's and B\"ottcher's Theorem when $f$ has an attracting or repelling integrable asymptotically conformal fixed point at $0$ with a certain control condition. More precisely, using the notation of \cite{Jia}, if we set
\[ \omega (r) = || \mu_f |_{B(0,r)} ||_{\infty} ,\]
then $f$ is asymptotically conformal at $0$ if
\[ \omega (r) \to 0\]
as $r\to 0^+$ and called integrable asymptotically conformal if
\[ \int_0^{r_0} \frac{ \omega_f(s) }{s} \: ds < \infty .\]
The control condition needed for Jiang's result when $f'(0) = \lambda$ is the requirement that there exist constants $r>0$ and $C>1$ such that
\[ \frac{1}{C} \leq \left | \frac{ f^n(z) }{\lambda^n z } \right | \leq C \]
for all $z\in \overline{B(0,r)}$ and all $n\geq 0$. Requiring a condition valid for all iterates is very strong, and one of the benefits to our approach is that we only require assumptions on the map and its asymptotic representative, not any of the iterates.

Jiang's approach to these results is to use the theory of holomorphic motions, which is not available to us in our setting. Jiang notes that there is a relationship between these conditions and the notion of $C^{1+\alpha}$-conformality. If $f$ is $C^{1+\alpha}$-conformal at $0$, this means that
\[ f(z) = f'(0)z + O( |z|^{1+\alpha} ).\]
This representation of $f$ is the spirit in which \eqref{eq:thm1eq2} in Theorem \ref{thm:1} arises, that is, there is a certain gap between the asymptotic representative and the error term. We emphasize that our results apply to a much wider class of quasiconformal maps than asymptotically conformal maps.

The only other analogous result for quasiconformal maps that the authors are aware of is a version of B\"ottcher's Theorem by the first named author and Fryer \cite{FF} for quasiconformal maps with constant complex dilatation. More precisely, if $h:\C \to \C$ is a non-degenerate real-linear map and $f(z) = [h(z)]^2+c$ then by \cite[Theorem 2.1]{FF} there is a neighbourhood of infinity on which $f$ may be conjugated to $[h(z)]^2$. The approach in this theorem can be extended to compositions of $h$ with higher degree polynomials, but these quasiregular maps have constant complex dilatation. Again we emphasize that the results here apply to a much wider class of maps.

\subsection{Structure of the paper}

In Section \ref{sec:prelims} we recall background material on quasiregular mappings in the plane, infinitesimal spaces, BIP maps and more. Section
\ref{sec:att} is the heart of the paper where Theorem \ref{thm:1} is proved. Section \ref{sec:rep} focusses on proving Theorem \ref{thm:2}. In Section \ref{sec:old} we prove Proposition \ref{prop:holo} and Proposition \ref{prop:holo2} on recovering the classical results from our results. Finally, in Section \ref{sec:conj} we give some concluding remarks on the uniqueness of the conjugacy.

\section{Preliminaries}
\label{sec:prelims}

\subsection{Quasiregular maps in the plane}

Let $U\subset \C$ be a domain. If $K\geq 1$, we say that a homeomorphism $f:U \to \C$ is $K$-quasiconformal if $f$ is in the Sobolev space $W^{1,2}_{\operatorname{loc}}(U)$ and 
\begin{equation}
\label{eq:pre0} 
|f_{\zbar} (z) | \leq \left ( \frac{K-1}{K+1} \right ) |f_z(z)| 
\end{equation}
for almost every $z\in U$. Here $f_{\zbar}$ and $f_z$ denote the complex derivatives defined by
\[ f_{\zbar} = \frac{1}{2} \left ( f_x + i f_y\right ), \quad f_z = \frac{1}{2} \left ( f_x - if_y \right ).\]
The smallest $K\geq 1$ for which \eqref{eq:pre0} holds is called the maximal dilatation of $f$ and gives a measure of how far $f$ is from a conformal map. When $K=1$, the right hand side of \eqref{eq:pre0} is $0$ and we recover the Cauchy-Riemann equations $f_{\zbar} \equiv 0$.

The complex dilatation of a quasiconformal map is defined to be
\[ \mu_f(z) = \frac{ f_{\zbar}(z) }{f_z(z) }.\]
From \eqref{eq:pre0} it follows that if $f$ is $K$-quasiconformal then $\mu_f$ lies in the closed ball centred at $0$ of radius $(K-1)/(K+1)$ in $L^{\infty}(U)$. Conversely, if $\mu$ lies in the unit ball of $L^{\infty}(U)$ there is a quasiconformal solution $f$ to the Beltrami differential equation $f_{\zbar} = \mu f_z$.
We also note that the Jacobian is given by $J_f(z) = |f_z|^2 - |f_{\zbar}|^2$.

We say that a quasiconformal map $f$ is asymptotically conformal at $z_0$ if for every $\epsilon >0$ there exists a neighbourhood $U$ of $z_0$ so that $||\mu_{f|_U} ||_{\infty} < \epsilon$.

A quasiregular map $f:U \to \C$ is just a quasiconformal map with the requirement of injectivity dropped. Alternatively, the Stoilow decomposition of a quasiregular map could be used as a definition. This states that a quasiregular map $f$ can always be decomposed as $f = g\circ h$ where $g$ is holomorphic and $h$ is quasiconformal. Informally speaking, $h$ takes care of the distortion and $g$ takes care of the branching. 

To define the branch set, we first recall that the local index $i(z,f)$ is the infimum of $\sup_w \operatorname{card}f^{-1}(w)\cap V$, where $V$ runs over all neighbourhoods of $z$. The branch set is then the set of $z$ for which $i(z,f) >1$. It follows from the Stoilow decomposition that the branch set of a quasiregular map in the plane is a discrete set. Consequently, a quasiregular map also has a well-defined complex dilatation.

The quasiregular version of Montel's Theorem as proved by Miniowitz, but specialized to dimension two for our purposes, is as follows.

\begin{theorem}[Theorem 4, \cite{Min}]
\label{thm:montel}
Let $\mathcal{F}$ be a family of $K$-quasiregular maps (without poles) in a plane domain $U$. If every $f\in \mathcal{F}$ omits a value $w_0 \in \C$, then $\mathcal{F}$ is a normal family.
\end{theorem}

Note the conclusion of this theorem says that if $(f_m)_{m=1}^{\infty}$ is any sequence in $\mathcal{F}$ then there exists a subsequence which either converges locally uniformly to a $K$-quasiregular map, or diverges locally uniformly to infinity. It is key in this result that the same $K$ holds uniformly over all maps in the family.

\subsection{Complex derivatives}

Computations and estimates involving complex derivatives will be crucial for our arguments. We recall some basic formulas we will need here.
The Chain Rules for complex derivatives are
\[ (f\circ g)_z(z) = f_z(g(z)) g_z(z) + f_{\zbar}(g(z)) \overline{ g_{\zbar}(z)}\]
and
\[ (f\circ g)_{\zbar}(z) = f_z(g(z)) g_{\zbar}(z) + f_{\zbar}(g(z)) \overline{g_z(z)}.\]
By applying these two formulas to $(f^{-1} \circ f)(z) = z$, we may obtain the formulas for the complex derivatives of an inverse:
\[ 
(f^{-1})_z(f(z)) = \frac{ \overline{f_z(z)} }{|f_z(z)|^2 - |f_{\zbar}(z) |^2} 
\]
and
\[ (f^{-1})_{\zbar} (f(z)) = \frac{ -f_{\zbar}(z)}{ |f_z(z)|^2 - |f_{\zbar}(z) |^2}.\]

Setting $r_f(z) = \overline{f_z(z)}/ f_z(z) = e^{-2i \arg f_z(z)}$ (note that $|r_f(z)| = 1$), the formula for the complex dilatation of a composition is
\[ \mu_{ g\circ f}(z) = \frac{ \mu_f(z) + r_f(z) \mu_g(f(z)) }{1+r_f(z) \overline{\mu_f(z)} \mu_g(f(z)) } .\]
The formula for the complex dilatation of an inverse is thus
\[ \mu_{f^{-1}}(f(z)) = - \frac{ \mu_f(z)}{r_f(z)}.\]

\subsection{Infinitesimal spaces}

In \cite{GMRV}, a generalization for the derivative of a quasiregular mapping $f:U\to \C$ at $x_0$ was given. For simplicity, we just recall these ideas in dimension two.
For $r >0$, let
\begin{equation}
\label{eq:fe} 
f_{r}(z) = \frac{ f(z_0  + r z) - f(z_0) }{\rho_f(r)},
\end{equation}
where $\rho_f(r)$ is the mean radius of the image of a sphere of radius $r$ centered at $z_0$ and given by
\begin{equation}
\label{eq:rho} 
\rho_f(r) = \left(\frac{| f(B(x_0,r)) |}{\pi }\right)^{1/2} .
\end{equation}
Here $|\cdot| $ denotes the usual area of a set. While each $f_{r}(x)$ is only defined on a ball centered at $0$ of radius $d(z_0,\partial U) / r$, when we consider limits as $r \to 0$, we obtain maps defined on all of $\C$. As each $f_{r}$ is a quasiregular mapping with the same bound on the maximal dilatation, it follows from Theorem \ref{thm:montel} that for any sequence $r_k \to 0$, there is a subsequence for which we do have local uniform convergence to some non-constant quasiregular mapping.

\begin{definition}
\label{def:genderiv}
Let $f:U \to \C$ be a quasiregular mapping defined on a domain $U\subset \C$ and let $z_0 \in \C$. A generalized derivative $g$ of $f$ at $z_0$ is defined by
\[ g(x) = \lim_{k\to \infty} f_{r_k}(x),\]
for some decreasing sequence $(r_k)_{k=1}^{\infty}$, whenever the limit exists. The collection of generalized derivatives of $f$ at $z_0$ is called the infinitesimal space of $f$ at $z_0$ and is denoted by $T(z_0,f)$.
\end{definition}

\begin{example}
If $\lambda \in \C \setminus  \{ 0 \}$ and $f(z) = \lambda z$, then $f_r(z) = e^{i\arg \lambda }z$ for any $r>0$. Consequently, $T(0,f)$ consists only of the map $g(z) = e^{i\arg \lambda} z$. If $f(z) = z^d$ for $d\in \N$, then $f_r(z) = z^d$ for any $r>0$. It follows that $T(0,f)$ consists only of the map $g(z) = z^d$.
\end{example}

These two examples illustrate the informal property that generalized derivatives maintain the shape of $f$ near $x_0$, but they lose information on the scale of $f$. In general, if a quasiregular map $f$ is real differentiable at $z_0\in \C$, then $T(z_0,f)$ consists only of a scaled multiple of the derivative of $f$ at $z_0$. 

\begin{definition}
\label{def:simple}
Let $f:U \to \C$ be quasiregular on a domain $U$ and let $z_0 \in U$. If the infinitesimal space $T(z_0,f)$ consists of only one element, then $f$ is called simple at $z_0$.
\end{definition}

If $f$ is simple at $z_0$, then $f$ has an asymptotic representation analogous to the first degree Taylor polynomial approximation of an analytic function. As usual, for simplicity we suppose $z_0 = f(z_0) = 0$. Then \cite[Proposition 4.7]{GMRV} states that as $z\to 0$, we have
\[ f(z) \sim \mathcal{D} (x) := \rho_f(|x|) g(x/ |x| ),\]
where $p(x) \sim q(x)$ as $x\to 0$ means
\[ |p(x) - q(x) | = o( |p(x)| + |q(x) | ).\]
The map $\mathcal{D}$ is called the asymptotic representative of $f$ at (in this case) $0$.

\subsection{The logarithmic transform and BIP maps}

The logarithmic transform is a highly useful tool that has been utilized many times in complex dynamics, first by Eremenko and Lyubich. We refer to Sixsmith's survey paper \cite{Six} for more on this and, in particular, \cite[Section 5]{Six} for a development of the logarithmic transform.

In our setting, if $f$ is quasiregular in a neighbourhood of $0$, $f(0)=0$ and $i(0,f) = d$, then as described in the introduction, we can define the logarithmic transform $\ft$ of $f$. Then $\ft$ is defined in a half-plane $\{z : \Re(z) < \log R \}$ for some $R>0$ and $\ft(z+2\pi i ) = \ft(z) + 2\pi d i$. Consequently, $\ft$ maps every line $\Im(z) = t$, for $t< \log R$, onto a curve $\gamma_t$ which is invariant under translation by $2\pi d i$.

As $f$ is quasiconformal, so is $\ft$. However, quasiconformal maps do not necessarily preserve local rectifiability of curves. For example, snowflake curves can be the images of lines under quasiconformal maps. We wish to rule out examples such as these, and so we recall the definition of bounded integrable parameterization maps (or BIP maps for short) from \cite[Definition 1.2]{FP} in dimension two.

\begin{definition}
\label{def:bip}
Let $R>0$ and let $f:B(0,R) \to \C$ be quasiregular with $f(0) = 0$. We say that $f$ is a bounded integrable parameterization map if there exists $P>0$ such that for every $t<\log R$, the parameterization $\gamma_t : [-\pi , \pi] \to \C$ defined by $\gamma_t(s) = \ft ( t+is)$ satisfies
\[ \int_{-\pi }^{\pi} |\gamma_t'(s)|^2 \: ds \leq P .\]
\end{definition}

An important point to note here is that this definition in \cite{FP} was given for quasiconformal maps. However, the definition and the proofs involved in the statement of Theorem \ref{thm:FP} below go through word for word if $i(0,f) > 1$. We leave the interested reader to check that this is so.
We also remark that the full definition for BIP maps in \cite{FP} for any dimension is necessarily more complicated than the two dimensional version we need here.
Recall that $\rho_f$ is the mean radius function. Then $\widetilde{\rho_f}$ is the logarithmic transform of $\rho_f$ defined by $\widetilde{\rho_f}(t) = \log \rho_f(e^t)$.

\begin{theorem}[Theorem 1.3 and Theorem 1.5, \cite{FP}]
\label{thm:FP}
Let $K\geq 1$, $R>0$ and let $f:B(0,R) \to \C$ be a simple BIP $K$-quasiregular map with $f(0) = 0$. Then there exists $L$ depending only on $K$ and $P$ such that $\widetilde{\rho_f}$ is $L$-bi-Lipschitz on $(-\infty, \log R)$ and the asymptotic representative $\dt$ is $L$-bi-Lipschitz for $\Re(z) < \log R$.
\end{theorem}

In particular, it follows from this result that $\mathcal{D}$ itself is quasiregular on a neighbourhood of $0$.
Note that for the case $f(z) = z^d$ we have $\ft(z) = dz$ and so
\[ \int_{-\pi }^{\pi} |\gamma_t'(s)|^2 \: ds = 2\pi d^2.\]
In particular, $P$ may depend on $i(0,f)$ and so we do not claim that the $L$ in Theorem \ref{thm:FP} is independent of $i(0,f)$.

\subsection{Sternberg's Linearization Theorem}

We will need the following important linearization theorem for smooth real functions. The statement below is the special case of Sternberg's Linearization Theorem for smooth functions, see \cite{Ste} as well as \cite{OFR} for a more recent survey of this and related results.

\begin{theorem}[\cite{Ste}]
\label{thm:ste}
Let $R,R'>0$ and let $f:[0,R] \to [0,R']$ be a $C^{\infty}$ function satisfying $f(0) =0$ and $f'(0) = \lambda$ with $\lambda \neq 0,1$. Then there exists $R_1>0$ and a $C^{\infty}$ map $\varphi$ which conjugates $f$ to $x\mapsto \lambda x$ in $[0,R_1]$, that is, 
\[ \varphi(f(x)) = \lambda \varphi(x)\]
for $x\in [0,R_1]$.
\end{theorem}

We just note that Sternberg only requires $f$ to be defined on a one-sided neighbourhood of the fixed point at $0$ to obtain this conjugacy, and not globally, as is sometimes stated in the literature.

\section{Attracting fixed points}
\label{sec:att}

In this section, we will prove Theorem \ref{thm:1}. The strategy is to take logarithmic transforms of $f$ and $\mathcal{D}$ to obtain $\ft$ and $\dt$ respectively and then define a sequence of maps via $\widetilde{\psi}_1 = \dt^{-1} \circ \ft$ and, for $k\geq 1$, 
$\widetilde{\psi}_{k+1} = \dt^{-1} \circ \widetilde{\psi}_k \circ \ft$. All of these maps will be defined on a half-plane of the form $\Re(z) < \log R$. 

We assume that $f$ and $\mathcal{D}$ are defined for $|z|<R$ so that $\ft$ and $\dt$ are defined for $\Re(z) < \log R$ and quasiconformal there. Moreover, if $d=i(0,f) = i(0,\mathcal{D})$ then both $\ft$ and $\dt$ satisfy
\[ \ft(z+2\pi i ) = \ft(z) + 2\pi d i, \quad \dt(z+2\pi i ) = \dt(z) + 2\pi d i.\]
It follows that $\widetilde{\psi}_1$ satisfies
\begin{align*}
\widetilde{\psi}_1(z+2\pi i ) &= \dt^{-1} ( \ft(z + 2\pi i ) ) \\
&= \dt^{-1} ( \ft(z) + 2\pi d i )\\
&= \dt^{-1} ( \ft(z) ) + 2\pi i \\
&= \widetilde{\psi}_1(z) + 2\pi i.
\end{align*}
This implies that there is a quasiconformal map $\psi_1$ whose logarithmic transform is $\widetilde{\psi}_1$. Continuing inductively, each element $\widetilde{\psi}_k$ of the sequence is the logarithmic transform of a quasiconformal map.

Let us assume for now that $f$ has a geometrically attracting fixed point at $0$ with parameter $\lambda \in (0,1)$. We will discuss the superattracting case at the end of this section.

The first step is to show that this sequence actually converges, and then the second step is to show that the limit is the logarithmic transform of an asymptotically conformal map which gives the desired conjugacy.

It will be important to control error terms in our analysis. To that end, we define
\[ \widetilde{E_k}(z) = \widetilde{\psi_k}(z)-z.\]
We also define $\widetilde{E}(z) = \ft (z) - \dt(z)$.

\subsection{Convergence of $\psi_k$}

Our goal in this section is to control $|\widetilde{E_k}(z)|$. To start, let us begin with analyzing $\widetilde{\psi_1}$. 
 
\begin{lemma}
 	\label{lm: E1tilde}
 	With the hypotheses of Theorem \ref{thm:1},  there exists $t_1<0$ such that when $\Re(z)<t_1,$ we have  
\[|\widetilde{E}_1(z)|<LT_1 e^{\alpha\Re(z)},\]
where $T_1$ is the constant from \eqref{eq:thm1eq1} in Theorem \ref{thm:1} and $L$ is the isometric distortion of $\dt$.
 \end{lemma}

 \begin{proof}
 	By the definitions of $\widetilde{E_1}$, $\psi_1$ and using \eqref{eq:thm1eq1}, we have 
 	\begin{align*}
 		|\widetilde{E_1}(z)|&=|\widetilde{\psi_1}(z)-z|\\
		&= | \dt^{-1} ( \ft(z)) - z)| \\
		&= |\dt^{-1}(\ft(z)) - \dt^{-1}(\dt(z)) | \\
		& \leq L| \ft(z) - \dt(z) | \\
		&\leq LT_1 e^{\alpha \Re(z)},
	\end{align*}
as required.
 \end{proof}

Next we combine the properties of $\ft$ with those of $\widetilde{E_1}$.

\begin{lemma}
	\label{lm: bounding transform E1 }
	With the assumptions of Theorem \ref{thm:1}, and the conclusions of Lemma \ref{lm: E1tilde}, we have
	\[  |\widetilde{E}( \ft ^k(z) ) | \leq T_1 \lambda^{\alpha k} e^{\alpha \Re(z)} \quad \text{ and } \quad  |\widetilde{E_1}(\ft^{k}(z))|\leq LT_1 \lambda^{\alpha k}e^{\alpha \Re(z)}, \]
	 for $\Re(z) < t_1$ and for all $k\in\N$. 
\end{lemma}

\begin{proof}
As $0$ is a geometrically attracting fixed point of $f$ with parameter $\lambda\in (0,1)$, it follows that
\begin{equation}
\label{eq: f_n^k bound} 
\Re(\ft^k(z)) \leq k \log \lambda + \Re(z)
\end{equation}
for all $k\in \N$.
Hence if $\Re(z) < \log R$, the same is true of $\Re \ft^k(z)$ for all $k\in \N$. It follows from Lemma \ref{lm: E1tilde} and \eqref{eq: f_n^k bound} that
	\begin{align*}
		|\widetilde{E}(\ft^k(z))|	&\leq T_1e^{\alpha \Re(\ft^k(z))}\\ 
		&\leq T_1e^{\alpha(k\log \lambda+\Re(z))}\\
		&=T_1\lambda^{\alpha k}e^{\alpha\Re(z)}.  
	\end{align*}
The result for $|\widetilde{E_1} ( \ft ^k (z)) |$ follows from this and the proof of Lemma \ref{lm: E1tilde}.
\end{proof}

We now come to the main step in this subsection.

\begin{proposition} 
	\label{prop: uniform bound}
	With the assumptions of Theorem \ref{thm:1}, there is a constant $t_2<0$ such that when $\Re(z)<t_2$ we have
	\begin{equation}
	\label{eq:etk}
	|\widetilde{E}_{k}(z)|\leq \frac{LT_1e^{\alpha \Re(z)}}{1-L\lambda^{\alpha}}. 
	\end{equation}
	Moreover, the sequence of maps $(\widetilde{\psi}_k)_{k=1}^\infty$ is uniformly convergent in the half-plane $\Re (z) < t_2$.
\end{proposition}

\begin{proof}
	First we  prove the uniform bound for $|\widetilde{E}_k(z)|$. Recalling that
\[ \widetilde{\psi}_{k+1}(z) = \dt^{-1} ( \widetilde{\psi}_k ( \ft (z))) = z + \widetilde{E}_{k+1}(z),\]
we have
	\begin{align*}
		\widetilde{\psi}_{k+1}(z)&=\dt^{-1}(\ft(z)+\widetilde{E}_k(\ft(z)))\\
		&=\dt^{-1}(\dt(z)+\widetilde{E}(z)+\widetilde{E}_k(\ft(z))).
	\end{align*}
Set $t_2 = \min \{t_1 , \log R \}$. Then as $\dt$ is $L$-bi-Lipschitz, for $\Re(z) < t_2$, we have
	\begin{align*}
		|\widetilde{E}_{k+1}(z)|&=|\widetilde{\psi}_{k+1}(z)-z|\\
		&= | \dt^{-1}( \widetilde{\psi}_k ( \ft(z))) - z|\\
		&=|\dt^{-1}[\dt(z)+\widetilde{E}(z)+\widetilde{E}_k(\ft(z))]-\dt^{-1}(\dt(z))| \\
		&\leq L|\widetilde{E}(z)|+L|\widetilde{E}_k(\ft(z))|. 
	\end{align*}
Using \eqref{eq:thm1eq1}, Lemma \ref{lm: bounding transform E1 } and that $L\lambda^{\alpha} <1$ from \eqref{eq:thm1eq4}, by induction we have 
	\begin{align*}
	|\widetilde{E}_{k+1}(z)|&\leq L|\widetilde{E}(z)| + L^2|\widetilde{E} (\ft (z)) | + L^2 |\widetilde{E}_{k-1}(\ft^2(z))| \\
	& \leq L | \widetilde{E}(z) | +  L^2|\widetilde{E} (\ft (z)) | + \ldots + L^{k} |\widetilde{E}(\ft^{k-1}(z)) | + L^k | \widetilde{E}_1(\ft^k(z)) |\\
	&\leq LT_1 e^{\alpha \Re(z)} + L^2 T_1 \lambda^{\alpha} e^{ \alpha \Re(z)} + \ldots + L^k T_1 \lambda^{\alpha(k-1)} e^{\alpha \Re(z)} + L^{k+1} T_1 \lambda^{\alpha k} e^{\alpha \Re(z)}\\
	&= LT_1 e^{\alpha \Re(z)} \left ( \sum_{m=0}^{k} L^m \lambda^{\alpha m} \right)\\
	&\leq \frac{ LT_1 e^{\alpha \Re(z)} }{1-L\lambda^{\alpha} }
\end{align*}
and this latter expression holds for all $k\in \N$. We conclude that $\widetilde{E}_k$ is uniformly bounded for $\Re(z)<t_2$.

For the second part of the proposition, we will use the Cauchy criterion for uniform convergence. Given $\epsilon >0$, then as $L\lambda^{\alpha} <1$, we can find $N\in \N$ such that
\begin{equation}
\label{eq:N}
\left[\frac{L}{1-L\lambda^{\alpha}}+1\right]T_1e^{\alpha t_2}(L\lambda^{\alpha})^{N-1}<\epsilon. 
\end{equation}
	Let $\Re(z) < t_2$. Then for $q\geq  p \geq N$ we have by using the fact that $\dt$ is $L$-bi-Lipschitz $p-1$ times that
	\begin{align*}
		\left|\widetilde{\psi}_{q}(z)-\widetilde{\psi}_{p}(z)\right|&=\left|\dt^{-1}\left(\dt(z)+\widetilde{E}(z)+\widetilde{E}_{q-1}(\ft(z))\right)-\dt^{-1}\left(\dt(z)+\widetilde{E}(z)+\widetilde{E}_{p-1}(\ft(z))\right)\right|\\
		&\leq L\left|\widetilde{E}_{q-1}(\ft(z))-\widetilde{E}_{p-1}(\ft(z))\right|\\
		&=L\left|\widetilde{\psi}_{q-1}(\ft(z))-\widetilde{\psi}_{p-1}(\ft(z)) \right|\\
		&\leq L^{p-1}\left|\widetilde{E}_{q-p+1}(\ft^{p-1}(z))-\widetilde{E}_1(\ft^{p-1}(z)) \right|\\
		&\leq L^{p-1}\left(|\widetilde{E}_{q-p+1}(\ft^{p-1}(z))|+|\widetilde{E}_1(\ft^{p-1}(z))| \right). 
	\end{align*}
	As $L\lambda^{\alpha}<1$ and $p\geq N$ it follows that $(L\lambda^{\alpha})^{p-1}\leq(L\lambda^{\alpha})^{N-1}$. Using this, 
Lemma \ref{lm: bounding transform E1 }, \eqref{eq:etk} and \eqref{eq:N}, we have 
	\begin{align*}
		\left|\widetilde{\psi}_{q}(z)-\widetilde{\psi}_{p}(z)\right|&\leq L^{p-1}\left(|\widetilde{E}_{q-p+1}(\ft^{p-1}(z))|+|\widetilde{E}_1(\ft^{p-1}(z))| \right)\\
		&\leq\frac{L^{p-1}LT_1e^{\alpha t_2}\lambda^{\alpha (p-1)}}{1-L\lambda^{\alpha}}+L^{p-1}T_1\lambda^{\alpha (p-1) }e^{\alpha t_2}\\
		&\leq \left[\frac{LT_1}{1-L\lambda^{\alpha}}+1\right]T_1e^{\alpha t_2}(L\lambda^{\alpha})^{N-1}\\
		&< \epsilon
	\end{align*}
	By the Cauchy criterion, we conclude that the sequence $(\widetilde{\psi}_k)_{k=1}^{\infty}$ is uniformly convergent for $\Re(z) < t_2$.
	
\end{proof}

\subsection{Uniform quasiconformality of $\psi_k$}

Throughout this subsection, $C_i$ will denote constants that depend only on the data $T_1,T_2,T_3$ and $L$ from the hypotheses of Theorem \ref{thm:1}.
Moreover, all the statements below are assumed to hold for $\Re(z) < \log R$.

\subsubsection{Preliminary estimates}

First, we give an alternative way of expressing the complex derivatives of an inverse in terms of $f_z$ and $\mu_f$.

\begin{lemma}
\label{lem:00}
We have
\[ (f^{-1})_z(f(z)) = \frac{ 1}{f_z(z) (1- |\mu_f(z)|^2)} \]
and
\[ (f^{-1})_{\zbar} (f(z)) = - \frac{ \mu_f(z) }{\overline{f_z(z)} ( 1-|\mu_f(z)|^2) } .\]
\end{lemma}

\begin{proof}
These follow directly from the definitions.
\end{proof}

\begin{lemma}
\label{lem:0}
For $\Re z < \log R$, we have
\[ \frac{1}{L} \leq | \dt_z(u) | \leq L. \]
\end{lemma}

\begin{proof}
As $\dt$ is $L$-bi-Lipschitz, we have
\[ |\dt_z(u)| \leq |\dt_z(u)| + |\dt_{\zbar}(u)| \leq L \]
and
\[ |\dt_z(u)| \geq |\dt_z(u)| - |\dt_{\zbar}(u)| \geq \frac{1}{L}.\]
\end{proof}

\begin{lemma}
\label{lem:1a}
For $\Re z < \log R$ we have
\[ \left | \frac{ \ft_z(z) }{ \dt_z(z) } \right | \leq 1 + C_1 e^{\nu \Re(z) }.\]
\end{lemma}

\begin{proof}
By \eqref{eq:thm1eq3} we have
\[ \left | \frac{ \ft_z(z) }{ \dt_z(z)} - 1 \right | \leq T_3|\dt_z(z) |^{-1} e^{\beta ' \Re(z) } .\]
Therefore by Lemma \ref{lem:0} and the reverse triangle inequality, the result follows.
\end{proof}

\begin{lemma}
\label{lem:1b}
For $\Re z < \log R$ we have
\[ \left | \frac{ \dt_z(z) }{\dt_z( \widetilde{\psi_k}(z) ) } \right | \leq 1 + C_2 e^{\nu \Re(z) } .\]
\end{lemma}

\begin{proof}
Using \eqref{eq:thm1eq1} and \eqref{eq:thm1eq3}, we have
\[ \left | \frac{ \dt_z(z) }{ \dt_z(\widetilde{\psi}_k(z)) } - 1 \right | \leq T_2\left ( T_1 e^{\alpha \Re(z)} \right )^{\beta} |  \dt_z(\widetilde{\psi}_k(z)) |^{-1}.\]
Therefore by Lemma \ref{lem:0} and the triangle inequality, the result follows.
\end{proof}

\begin{lemma}
\label{lem:2}
With the assumptions of Theorem \ref{thm:1}, there exists a constant $C_3$ such that
\[ |\mu_{\dt}(u) - \mu_{\dt}(v) | \leq C_3 |u-v|^{\beta} .\]
\end{lemma}

\begin{proof}
Via the triangle inequality, the fact $\dt$ is $L$-bi-Lipschitz, \eqref{eq:thm1eq2}, and that $||\mu_{\dt} ||_{\infty} < 1$, we have
\begin{align*}
| \mu_{\dt}(u)  - \mu_{\dt}(v)  | &= \left |  \frac{ \dt_{\zbar}(u)}{\dt_z(u)} - \frac{ \dt_{\zbar}(v) }{\dt_z(v) } \right | \\
&= \frac{ | \dt_{\zbar}(u)\dt_z(v) - \dt_{\zbar}(v)\dt_z(u) | }{|\dt_z(u) | \cdot | \dt_z(v)| } \\
&\leq \frac{ | \dt_{\zbar}(u)\dt_z(v)  - \dt_z(u) \dt_{\zbar}(u)|}{|\dt_z(u) | \cdot | \dt_z(v)| } + 
\frac{ | \dt_z(u) \dt_{\zbar}(u) - \dt_{\zbar}(v)\dt_z(u)  |}{|\dt_z(u) | \cdot | \dt_z(v)| } \\
&\leq  \frac{ || \mu_{\dt} ||_{\infty} |\dt_z(v) - \dt_z(u) | }{|\dt_z (v) |}  +  \frac{ |\dt_{\bar{z}}(u) - \dt_{\bar{z}}(v) | }{|\dt_z(v)|} \\
&< 2LT_2 | u-v|^{\beta},
\end{align*}
as required.
\end{proof}

\begin{lemma}
\label{lem:2b}
With the assumptions of Theorem \ref{thm:1}, there exists a constant $C_4$ such that
\[ | \mu_{\ft}(z) - \mu_{\dt}(z)| \leq C_4 e^{\nu \Re(z)}.\]
\end{lemma}

\begin{proof}
Via the triangle inequality, the fact that $\dt$ is bi-Lipschitz and \eqref{eq:thm1eq3}, we have
\begin{align*}
|\mu_{\ft}(z) - \mu_{\dt}(z) | &= \left | \frac{ \ft_{\zbar}(z) }{\ft_z(z)} - \frac{ \dt_{\zbar}(z)} {\dt_z(z)} \right | \\
& \leq \frac{ | \ft_{\zbar}(z) \dt_z(z) - f_{\zbar}(z) f_z(z) |}{ |\ft_z(z)| \cdot |\dt_z(z)| } + \frac{ | \ft_{\zbar}(z) \ft_z(z) - \dt_{\zbar}(z) \ft_z(z)| }{|\ft_z(z)| \cdot |\dt_z(z)| } \\
&\leq \frac{ || \mu_{\ft} ||_{\infty} |\dt_z(z) - \ft_z(z)| }{ |\dt_z(z)|} + \frac{| \ft_{\zbar}(z) - \dt_{\zbar}(z)| }{|\dt_z(z)|} \\
&\leq 2LT_3 e^{\nu \Re(z)},
\end{align*}
as required.
\end{proof}

\begin{corollary}
\label{cor:2}
With the assumptions of Theorem \ref{thm:1}, there exist constants $C_5,C_6$ such that for all $k\in \N$ we have
\[ |\mu_{\ft}(z) - \mu_{\dt}( \widetilde{\psi}_k(z) ) | \leq C_5 e^{\nu \Re(z) }\]
and
\[ \left | \frac{  1 - \overline{ \mu_{\ft}(z) } \mu_{\dt} ( \widetilde{ \psi}_k (z) ) }{1 - | \mu_{\dt} ( \widetilde{ \psi}_k (z) ) |^2 } -1 \right | \leq C_6 e^{\nu \Re(z) }.\]
\end{corollary}

\begin{proof}
The first part follows from the triangle inequality, Lemma \ref{lem:2} and Lemma \ref{lem:2b}. For the second part, we use the first part to obtain
\begin{align*}  
\left | \frac{  1 - \overline{ \mu_{\ft}(z) } \mu_{\dt} ( \widetilde{ \psi}_k (z) ) }{1 - | \mu_{\dt} ( \widetilde{ \psi}_k (z) ) |^2 } -1 \right | &=
\left | \frac{  \mu_{\dt} ( \widetilde{ \psi}_k (z) ) [ \overline{ \mu_{\dt} ( \widetilde{ \psi}_k (z) )} - \overline{\mu_{\ft}(z) } ] }{1 - | \mu_{\dt} ( \widetilde{ \psi}_k (z) ) |^2 }  \right | \\
&\leq \left ( \frac{ ||\mu_{\dt} ||_{\infty} }{ 1 - ||\mu_{\dt} ||_{\infty }^2 } \right ) C_5 e^{\nu \Re(z)},
\end{align*}
as required.
\end{proof}

\subsubsection{Complex derivatives of the base case $\widetilde{\psi}_1$}

Our strategy to estimate the complex dilatation of $\widetilde{\psi}_k$ will involve an induction. In this subsection we prove the base case.
Recall that $\widetilde{\psi}_1 = \dt^{-1} \circ \ft$.

\begin{lemma}
\label{lem:3}
With the assumptions of Theorem \ref{thm:1}, there exists a constant $C_7$ such that
\[ | (\widetilde{\psi}_1)_z(z) - 1 | \leq C_7e^{\nu \Re(z)} .\]
\end{lemma}

\begin{proof}
First, by using the Chain Rule and Lemma \ref{lem:00}, we have
\begin{align*}
(\widetilde{\psi}_1)_z(z) &= (\dt^{-1})_z( \dt \circ \widetilde{\psi}_1(z) ) \ft_z(z) + (\dt^{-1})_{\zbar}( \dt \circ \widetilde{\psi}_1(z) ) \overline{\ft_{\zbar}(z) }\\
&= \frac{ \ft_z(z) }{ \dt_z( \widetilde{ \psi}_1 (z) )  ( 1 - |\mu_{\dt}( \widetilde{\psi}_1(z) ) |^2 ) } - \frac{ \mu _{\dt} ( \widetilde{\psi}_1(z)  ) \overline{ f_{\zbar}(z) } }{ \overline{ \dt_z ( \widetilde{ \psi}_1 (z) ) }  ( 1 - |\mu_{\dt}( \widetilde{\psi}_1(z) ) |^2 )} \\
&= \frac{ 1}{  ( 1 - |\mu_{\dt}( \widetilde{\psi}_1(z) ) |^2 ) } \left ( \frac{ \ft_z(z) }{\dt_z( \widetilde{ \psi}_1 (z) )} - \overline{ \mu_{\ft}(z) }  \mu _{\dt} ( \widetilde{\psi}_1(z)  ) \overline{ \left ( \frac{ \ft_z(z) }{\dt_z( \widetilde{ \psi}_1 (z) )} \right) } \right ).
\end{align*}
As
\[ \frac{ \ft_z(z) }{\dt_z( \widetilde{ \psi}_1 (z) )} = \frac{ \ft_z(z) }{ \dt_z(z) } \cdot \frac{ \dt_z(z) }{\dt_z( \widetilde{ \psi}_1 (z) )} ,\]
by Lemma \ref{lem:1a} and Lemma \ref{lem:1b} there exists a constant $C_8$ such that
\[ \frac{ \ft_z(z) }{\dt_z( \widetilde{ \psi}_1 (z) )}  = 1 + \epsilon_1(z),\]
where
\begin{equation}
\label{eq:11} 
|\epsilon_1(z)| \leq C_8 e^{\nu \Re(z) }.
\end{equation}
Therefore,
\begin{align*} 
(\widetilde{\psi}_1)_z(z)  &= \frac{ 1}{  ( 1 - |\mu_{\dt}( \widetilde{\psi}_1(z) ) |^2 ) } \left ( 1 + \epsilon_1(z) -  \overline{ \mu_{\ft}(z) }  \mu _{\dt} ( \widetilde{\psi}_1(z)  ) ( 1+ \overline{\epsilon_1(z) } ) \right ) \\
&=  \frac{  1 - \overline{ \mu_{\ft}(z) } \mu_{\dt} ( \widetilde{ \psi}_1 (z) ) }{1 - | \mu_{\dt} ( \widetilde{ \psi}_1 (z) ) |^2 } + \frac{ \epsilon_1(z) -   \overline{ \mu_{\ft}(z) }  \mu _{\dt} ( \widetilde{\psi}_1(z)  ) \overline{\epsilon_1(z)} }{1 - | \mu_{\dt} ( \widetilde{ \psi}_1 (z) ) |^2 }.
\end{align*}
It follows from Corollary \ref{cor:2} and \eqref{eq:11} that
\begin{align*}
|(\widetilde{\psi}_1)_z(z) - 1 | & \leq \left | \frac{  1 - \overline{ \mu_{\ft}(z) } \mu_{\dt} ( \widetilde{ \psi}_1 (z) ) }{1 - | \mu_{\dt} ( \widetilde{ \psi}_1 (z) ) |^2 } -1 \right | + |\epsilon_1(z)| \cdot \left ( \frac{ 1 + ||\mu_{\ft}||_{\infty} || \mu_{\dt} ||_{\infty} }{1- ||\mu_{\dt}||_{\infty}^2 } \right ) \\
&\leq C_6e^{\nu \Re(z) } + C_8 \left ( \frac{ 1 + ||\mu_{\ft}||_{\infty} || \mu_{\dt} ||_{\infty} }{1- ||\mu_{\dt}||_{\infty}^2 } \right )  e^{\nu\Re(z)}
\end{align*}
as required.
\end{proof}

\begin{lemma}
\label{lem:4}
With the assumptions of Theorem \ref{thm:1}, there exists a constant $C_9$ such that
\[ | (\widetilde{\psi}_1)_{\zbar}(z)  | \leq C_9 e^{\nu \Re(z)} .\]
\end{lemma}

\begin{proof}
Again by the Chain Rule and Lemma \ref{lem:00}, we have
\begin{align*}
(\widetilde{\psi}_1)_{\zbar}(z) &= (\dt^{-1})_z (\dt \circ \widetilde{\psi}_1(z) ) \ft_{\zbar}(z) + (\dt^{-1})_{\zbar}(\dt \circ \widetilde{\psi}_1(z)) \overline{\ft_z(z)} \\
&=   \frac{ \ft_{\zbar}(z) }{ \dt_z( \widetilde{ \psi}_1 (z) )  ( 1 - |\mu_{\dt}( \widetilde{\psi}_1(z) ) |^2 ) } - \frac{ \mu _{\dt} ( \widetilde{\psi}_1(z)  ) \overline{ f_{z}(z) } }{ \overline{ \dt_z ( \widetilde{ \psi}_1 (z) ) }  ( 1 - |\mu_{\dt}( \widetilde{\psi}_1(z) ) |^2 )}\\
&= \frac{ 1}{  ( 1 - |\mu_{\dt}( \widetilde{\psi}_1(z) ) |^2 ) } \left (  \frac{ \ft_z(z) \mu_{\ft}(z) }{ \dt_z( \widetilde{ \psi}_1 (z) ) } -  \frac{ \mu _{\dt} ( \widetilde{\psi}_1(z)  ) \overline{ f_{z}(z) } }{ \overline{ \dt_z ( \widetilde{ \psi}_1 (z) ) } } \right ) \\
&= \frac{ 1}{  ( 1 - |\mu_{\dt}( \widetilde{\psi}_1(z) ) |^2 ) } \left ( (1+\epsilon_1(z) ) \mu_{\ft}(z) - (1+\overline{\epsilon_1(z) } )\mu_{\dt} (\widetilde{\psi}_1(z) ) \right ).
\end{align*}
By \eqref{eq:11} and Corollary \ref{cor:2}, we conclude that
\begin{align*} 
|(\widetilde{\psi}_1)_{\zbar}(z)| &\leq \frac{ | \mu_{\ft}(z) - \mu_{\dt}( \widetilde{\psi}_1(z)) | }{ 1 - |\mu_{\dt}( \widetilde{\psi}_1(z) ) |^2 }
+ |\epsilon_1(z)| \cdot \left (  \frac{ | \mu_{\ft}(z)| + |\mu_{\dt}( \widetilde{\psi}_1(z)) | }{  1 - |\mu_{\dt}( \widetilde{\psi}_1(z) ) |^2 } \right ) \\
&\leq \frac{ C_5 e^{\nu \Re(z)} }{1 - ||\mu_{\dt} ||_{\infty}^2} + C_8 \left ( \frac{ ||\mu_{\ft}||_{\infty} + ||\mu_{\dt}||_{\infty} } {1 - ||\mu_{\dt} ||_{\infty}^2}  \right )e^{\nu \Re(z)}
\end{align*}
as required.
\end{proof}

Denote by $\mu_1$ the complex dilatation of $\widetilde{\psi}_1$.

\begin{corollary}
\label{cor:1}
There exists a constant $C_{10}$ such that
\[ |\mu_1(z)| \leq C_{10}e^{\nu\Re(z)} \]
\end{corollary}

\begin{proof}
By Lemma \ref{lem:3} and Lemma \ref{lem:4}, we have
\begin{align*}
| \mu_1(z)| &= \frac{ |(\widetilde{\psi}_1)_{\zbar}(z) |}{|(\widetilde{\psi}_1)_z(z) | } \\
&\leq  \frac{ C_7e^{\nu\Re(z)}  }{ 1 - C_9e^{\nu\Re(z)} } \\
&\leq C_{10}e^{\nu\Re(z)} 
\end{align*}
as required.
\end{proof}

\subsubsection{The inductive step}

In this subsection, we will show that bounds on the complex derivatives on $\widetilde{\psi}_k$ propagate to $\widetilde{\psi}_{k+1}$.
Recall from \eqref{eq:thm1eq4} that
\begin{equation}
\label{eq:6}
\lambda < K(\dt)^{-1/\nu} = \left ( \frac{ 1 - ||\mu_{\dt}||_{\infty}}{1 + ||\mu_{\dt} ||_{\infty} } \right ) ^{1/\nu}.
\end{equation}
Fix $\epsilon > 0$ such that
\begin{equation}
\label{eq:7}
\lambda^{\nu} K(\dt) < 1- \epsilon.
\end{equation}

\begin{lemma}
\label{lem:7}
We may choose $R<1$ such that for all $k \in \N$ and $\Re(z) < \log R$ we have
\[ \left ( 1 + C_8e^{\nu \log R} \right ) \cdot \left ( \frac{ 1 + |\mu_{\ft}(z)| }{ 1 + |\mu_{\dt}( \widetilde{\psi}_k(z) )| } \right ) < 1 + \frac{ \epsilon}{2}. \]
\end{lemma}

\begin{proof}
By Corollary \ref{cor:2}
\begin{align*}
 \frac{ 1 + |\mu_{\ft}(z) |}{1 + |\mu_{\dt}(\widetilde{\psi}_k(z) )|  }  &\leq \left |\frac{ 1 + |\mu_{\dt}(\widetilde{\psi}_k(z))| + |\mu_{\ft}(z)| -  |\mu_{\dt}(\widetilde{\psi}_k(z))| }{1 + |\mu_{\dt}(\widetilde{\psi}_k(z))|} \right | \\
&\leq 1 +  C_5 e^{\nu \Re(z) } 
\end{align*}
from which the lemma follows.
\end{proof}

Next choose a constant $\ct$ such that
\begin{equation}
\label{eq:8}
\ct > \frac{ 2\max\{ C_7,C_9\}  }{\epsilon }
\end{equation}
so that for some $k\in \N$ we have
\begin{equation}
\label{eq:5}
\max \{ | (\widetilde{\psi}_k)_z(z)-1|,  | (\widetilde{\psi}_k)_{\zbar}(z)|  \} \leq \ct e^{\nu\Re(z)} .
\end{equation}

First, we deal with the $z$-derivative of $\widetilde{\psi}_{k+1}$.

\begin{lemma}
\label{lem:5}
With the assumptions above, we have
\[ | (\widetilde{\psi}_{k+1})_z(z)-1| \leq \ct e^{\nu\Re(z)} .\]
\end{lemma}

\begin{proof}
By the Chain Rule, we have
\begin{align*}
(\widetilde{\psi}_{k+1})_z(z) &= (\dt^{-1} \circ \widetilde{\psi}_k \circ \ft )_z(z) \\
&= (\dt^{-1})_z(\widetilde{\psi}_k(\ft(z)) (\widetilde{\psi}_k \circ \ft )_z(z) + (\dt^{-1})_{\zbar}(\widetilde{\psi}_k(\ft(z)) \overline{ (\widetilde{\psi}_k \circ \ft)_{\zbar}(z) }\\
&= \frac{ (\widetilde{\psi}_k \circ \ft )_z(z) }{ \dt_z( \widetilde{\psi}_{k+1} (z)) (1 - | \mu_{\dt}( \widetilde{\psi}_{k+1}(z) ) |^2 ) }
- \frac{  \mu_{\dt}( \widetilde{\psi}_{k+1}(z) ) \overline{ (\widetilde{\psi}_k \circ \ft )_{\zbar}(z) } }{ \overline{ \dt_z ( \widetilde{\psi}_{k+1 } (z)) }  (1 - | \mu_{\dt}( \widetilde{\psi}_{k+1}(z) ) |^2 ) }
\end{align*}
For convenience, write $(\widetilde{\psi}_k)_z (z) = 1 + p_k(z)$ and $(\widetilde{\psi}_k)_{\zbar}(z) = q_k(z)$, with bounds on $|p_k|,|q_k|$ from \eqref{eq:5}. Then by the Chain Rule
\[ (\widetilde{\psi}_k \circ \ft )_z(z) = (1+p_k(\ft(z))) \ft_z(z) + q_k(\ft(z)) \overline{\ft_{\zbar}(z) }\]
and 
\[ (\widetilde{\psi}_k \circ \ft )_{\zbar}(z) = (1+p_k(\ft(z))) \ft_{\zbar}(z) + q_k(\ft(z)) \overline{ \ft_z(z) }.\]
Putting this all together, we obtain
\begin{align*} 
(\widetilde{\psi}_{k+1})_z(z) &=  \left ( \frac{ (1+p_k(\ft(z))) \ft_z(z) + q_k(\ft(z)) \overline{\ft_{\zbar}(z) } }{\dt_z( \widetilde{\psi}_{k+1} (z)) (1 - | \mu_{\dt}( \widetilde{\psi}_{k+1}(z) ) |^2 )} - \frac{ \mu_{\dt}( \widetilde{\psi}_{k+1}(z) )[( 1+ \overline{p_k(\ft(z)) } )\overline{\ft_{\zbar}(z)} + \overline{q_k(\ft(z))} \ft_z(z) ]  }{ \overline{ \dt_z ( \widetilde{\psi}_{k+1}  (z)) }(1 - | \mu_{\dt}( \widetilde{\psi}_{k+1}(z) ) |^2 )} \right )\\
&= \frac{1}{ (1 - | \mu_{\dt}( \widetilde{\psi}_{k+1}(z) ) |^2 ) } \left ( \frac{ \ft_z(z) }{\dt_z( \widetilde{\psi}_{k+1 }(z)) } - \frac {\mu_{\dt}( \widetilde{\psi}_{k+1}(z)) \overline{ \ft_{\zbar}(z) } } { \overline{ \dt_z( \widetilde{ \psi}_{k+1} (z)) } } \right )\\
&+ \frac{1}{ (1 - | \mu_{\dt}( \widetilde{\psi}_{k+1}(z) ) |^2 ) } \left (   \frac{p_k(\ft(z)) \ft_z(z)}{\dt_z( \widetilde{\psi}_{k+1 }(z)) } -  \frac { \overline{ p_k(\ft(z)) } \mu_{\dt}( \widetilde{\psi}_{k+1}(z)) \overline{\mu_{\ft}(z) }\overline{ \ft_z(z) }} { \overline{ \dt_z( \widetilde{ \psi}_{k+1} (z)) } }   \right )\\
&+ \frac{1}{ (1 - | \mu_{\dt}( \widetilde{\psi}_{k+1}(z) ) |^2 ) } \left ( \frac{ q_k(\ft(z)) \overline{\mu_{\ft} (z) } r_{\ft}(z) \ft_z(z) }{\dt_z( \widetilde{\psi}_{k+1}(z) ) } -  \frac{ \overline{q_k(\ft(z))} \mu_{\dt}( \widetilde{\psi}_{k+1} (z)) \overline{r_{\ft}(z)} \overline{ \ft_z(z) } }{\overline{\dt_z(\widetilde{\psi}_{k+1} (z)) } }  \right )
\end{align*}
Using the same idea as in the proof of Lemma \ref{lem:3}, the first term here is within $C_7 e^{\nu \Re(z) }$ of $1$. Next, we will repeatedly use the fact that via Lemma \ref{lem:1a} and Lemma \ref{lem:1b} we may write
\[ \frac{ \ft_z(z) }{ \dt_z( \widetilde{\psi_{k+1}} (z) ) } = 1+\epsilon_{k+1}(z), \quad  |\epsilon_{k+1}| \leq C_8 e^{\nu \Re(z)}\]
uniformly in $k$.
Next, by \eqref{eq: f_n^k bound} and \eqref{eq:5}, 
\begin{align*} 
\max \{ |p_k(\ft(z))| , |q_k(\ft (z) )| \} &\leq \ct e^{\nu \Re \ft(z) }\\
&\leq \ct \lambda^{\nu} e^{\nu \Re(z)} .
\end{align*}
Then the absolute value of the second term above is bounded above by
\[ \frac{ \ct \lambda^{\nu} e^{\nu \Re(z) } (1+C_8e^{\nu \Re(z) }) (1+|\mu_{\dt}( \widetilde{\psi}_{k+1}(z) )| |\mu_{\ft}( z )| )}{1- |\mu_{\dt}( \widetilde{\psi}_{k+1}(z) )|^2 }\]
and the absolute value of the third term is bounded above by
\[ \frac{ \ct \lambda^{\nu} e^{\nu \Re(z) } (1+C_8e^{\nu \Re(z) }) (|\mu_{\dt}( \widetilde{\psi}_{k+1}(z) )|+ |\mu_{\ft}(z )|)}{1- |\mu_{\dt}( \widetilde{\psi}_{k+1}(z) )|^2 }.\]

Finally, by Lemma \ref{lem:7}, \eqref{eq:7} and \eqref{eq:8} we have
\begin{align*}
| (\widetilde{\psi}_{k+1})_z(z)-1 | &\leq \left ( C_7 + \frac{ \ct \lambda^{\nu} (1+C_8 e^{\nu \Re(z)} ) ( 1+ |\mu_{\ft}( z)| )(1+|\mu_{\dt}( \widetilde{\psi}_{k+1}(z) )| )  }{1-|\mu_{\dt}( \widetilde{\psi}_{k+1}(z) )|^2} \right) e^{\nu \Re(z)} \\
&\leq \left ( \frac{\epsilon}{2} + \frac{(1 - \epsilon)}{K(\dt)} (1+C_8 e^{\nu M} ) \left ( \frac{ 1 + |\mu_{\ft}(z)| }{1- |\mu_{\dt}(\widetilde{\psi}_{k+1}(z))| } \right ) \right ) \ct e^{\nu \Re(z)}\\
&\leq \left ( \frac{ \epsilon}{2} + (1-\epsilon)(1+C_8 e^{\nu M} ) \left (  \frac{ 1 + |\mu_{\ft}(z) |}{1 + |\mu_{\dt}(\widetilde{\psi}_k(z) )|  } \right ) \right ) \ct e^{\nu \Re(z)}\\
&\leq \left ( \frac{\epsilon}{2} + (1-\epsilon) (1+\epsilon/2) \right ) \ct e^{\nu \Re(z)}\\
&= \left ( 1 - \frac{\epsilon^2}{2} \right ) \ct e^{\nu \Re(z) }\\
&< \ct e^{\nu \Re(z)}
\end{align*}
as required.
\end{proof}

Second, we deal with the $\zbar$-derivative of $\widetilde{\psi}_{k+1}$.

\begin{lemma}
\label{lem:6}
With the assumptions above, we have 
\[ | ( \widetilde{\psi}_{k+1} )_{\zbar} (z) | \leq \ct e^{\nu \Re(z) } .\]
\end{lemma}

\begin{proof}
By the Chain Rule and using the notation $p_k,q_k$ introduced above, we have
\begin{align*}
(\widetilde{\psi}_{k+1})_{\zbar}(z) &= (\dt^{-1} \circ \widetilde{\psi}_k \circ \ft )_{\zbar}(z) \\
&= (\dt^{-1})_z(\widetilde{\psi}_k(\ft(z)) (\widetilde{\psi}_k \circ \ft )_{\zbar}(z) + (\dt^{-1})_{\zbar}(\widetilde{\psi}_k(\ft(z)) \overline{ (\widetilde{\psi}_k \circ \ft)_z(z) }\\
&= \frac{ (1+p_k(\ft(z)) )\ft_{\zbar}(z) + q_k(\ft(z)) \overline{\ft_z(z)} }{ \dt_z( \widetilde{\psi}_{k+1}(z)  ) (1- |\mu_{\dt} (\widetilde{\psi}_{k+1}(z) ) |^2 )}
- \frac{ \mu_{\dt} ( \widetilde{\psi}_{k+1}(z)) [ (1 + \overline{ p_k(\ft(z)) } ) \overline{\ft_z(z) } + \overline{q_k(\ft(z))} f_{\zbar}(z) ] }{  \overline{ \dt_z( \widetilde{\psi}_{k+1}(z) ) }(1- |\mu_{\dt} (\widetilde{\psi}_{k+1}(z)) |^2 ) }\\
&= \frac{1}{1-|\mu_{\dt}(\widetilde{\psi}_{k+1}(z)) |^2} \left ( \frac{ \mu_{\ft}(z) f_z(z) }{ \dt_z ( \widetilde{\psi}_{k+1}(z) ) } - \frac{ \mu_{\dt} (\widetilde{\psi}_{k+1}(z) ) \overline{f_z(z)}}{ \overline{ \dt_z (\widetilde{\psi}_{k+1}(z) ) } } \right ) \\
&+ \frac{1}{1-|\mu_{\dt}(\widetilde{\psi}_{k+1}(z)) |^2} \left ( \frac{ p_k(\ft(z)) \mu_{\ft}(z) f_z(z) }{\dt_z ( \widetilde{\psi}_{k+1}(z) )} - \frac{ \overline{p_k(\ft(z)) } \mu_{\dt} (\widetilde{\psi}_{k+1 }(z)) \overline{f_z(z)} }{  \overline{ \dt_z (\widetilde{\psi}_{k+1}(z) ) } }   \right ) \\
&+ \frac{1}{1-|\mu_{\dt}(\widetilde{\psi}_{k+1}(z)) |^2} \left ( \frac{ q_k(\ft(z)) r_{\ft}(z) f_z(z) }{\dt_z ( \widetilde{\psi}_{k+1}(z) )} - \frac{ \overline{q_k(\ft(z)) } \mu_{\dt} (\widetilde{\psi}_{k+1 }(z)) \mu_{\ft}(z) r_{\ft}(z) \overline{\ft_z(z)} }{  \overline{ \dt_z (\widetilde{\psi}_{k+1}(z) ) } }   \right )
\end{align*}
Following the same idea as the proof of Lemma \ref{lem:5}, the first term may be bounded by $C_9e^{\nu \Re(z)}$ just as the analogous term was in Lemma \ref{lem:4}. The other two terms may be bounded as in Lemma \ref{lem:5} from which we again obtain
\begin{align*}
|(\widetilde{\psi}_{k+1})_{\zbar}(z)| &\leq  \left ( C_9 + \frac{ \ct \lambda^{\nu} (1+C_8 e^{\nu \Re(z)} ) ( 1+ |\mu_{\ft}( z)| )(1+|\mu_{\dt}( \widetilde{\psi}_{k+1}(z) )| )  }{1-|\mu_{\dt}( \widetilde{\psi}_{k+1}(z) )|^2} \right) e^{\nu \Re(z)} 
\\
&\leq \ct e^{\nu \Re(z)},
\end{align*}
using Lemma \ref{lem:7}, \eqref{eq:7} and \eqref{eq:8}.
\end{proof}

For $k\in \N$, define $\mu_k$ to be the complex dilatation of $\widetilde{\psi}_k$.

\begin{corollary}
\label{cor:1a}
There exists a constant $C_{11}$ such that for all $k\in \N$ and $\Re(z) < \log R$ we have
\[ |\mu_k(z)| \leq C_{11}e^{\nu\Re(z)} \]
\end{corollary}

\begin{proof}
By Lemma \ref{lem:5} and Lemma \ref{lem:6}, we have
\begin{align*}
| \mu_k(z)| &= \frac{ |(\widetilde{\psi}_k)_{\zbar}(z) |}{|(\widetilde{\psi}_k)_z(z) | } \\
&\leq  \frac{ \ct e^{\nu\Re(z)}  }{ 1 - \ct e^{\nu\Re(z)} } \\
&\leq C_{11}e^{\nu\Re(z)} 
\end{align*}
as required.
\end{proof}

\subsection{Completing the proof}

\begin{proof}[Proof of Theorem \ref{thm:1}]

Recall that we are initially assuming that $0$ is a geometrically attracting fixed point of $f$ with parameter $\lambda$ satisfying \eqref{eq:thm1eq4}.

From Proposition \ref{prop: uniform bound}, the sequence $\widetilde{\psi}_k$ converges uniformly for $\Re(z) < \log R$ to some continuous map $\widetilde{\psi}$. Moreover, as $\widetilde{\psi}_k(z) = z + \widetilde{E_k}(z)$, the uniform bound for $|\widetilde{E_k}|$ given by \eqref{eq:etk} implies that $\widetilde{\psi}_k$ omits a half-plane of the form $\Re(z) > S $ for all $k$.

Combining this observation with the uniform bound on $\mu_k$ from Corollary \ref{cor:1a} means we can apply Montel's Theorem, Theorem \ref{thm:montel}, to conclude that the limit function $\widetilde{\psi}$ is quasiconformal and also satisfies
\[ |\mu_{\widetilde{\psi}}(z) | \leq C_{11} e^{\nu \Re(z)} .\]
As each $\widetilde{\psi}_k$ satisfies $\widetilde{\psi}_k(z + 2\pi i ) = \widetilde{\psi}_k(z) + 2\pi i $, the same is true of $\widetilde{\psi}$. Hence there is an asymptotically conformal quasiconformal map $\psi$ whose logarithmic transform is $\widetilde{\psi}$.
Now, as
\begin{align*}
 \dt \circ \widetilde{\psi}_k &= \dt^{-k+1} \circ \ft ^k \\
&= \dt^{-k+1} \circ \ft^{ k-1} \circ \ft\\
&= \widetilde{\psi}_{k-1} \circ \ft,
\end{align*}
when we let $k\to \infty$ we conclude that
\[ \dt \circ \widetilde{\psi} = \widetilde{\psi} \circ \ft .\]
Undoing the logarithmic transform here shows that
\[ \mathcal{D} \circ \psi = \psi \circ f \]
as required.

If $0$ is a superattracting fixed point of $f$, then there exists $R'>0$ such that $|f(z)| < \lambda |z|$ for $|z|<R'$ where $\lambda$ satisfies \eqref{eq:thm1eq4}. We may then apply the argument above to obtain the conclusion.
\end{proof}

\section{Repelling fixed points}
\label{sec:rep}

In this section, we show that if the hypotheses of Theorem \ref{thm:2} hold, then the inverses of the logarithmic transforms $\ft^{-1}$ and $\dt^{-1}$ satisfy the hypotheses of Theorem \ref{thm:1}, from which the required conjugacy follows.

\subsection{Condition (a)} First observe that $|f(z)| > \lambda |z|$ implies that
$\Re \ft(z) > \log \lambda + \Re z$ and hence
\begin{equation}
\label{eq:rep1}
\Re \ft^{-1}(z) < \Re(z) + \log (1/\lambda ).
\end{equation}
If
\[ |\ft(z) - \dt(z) | < S_1 e^{\alpha \Re(z)}\]
for $\Re(z) < \log R$ then by \eqref{eq:rep1} we have
\begin{align*}
| \ft^{-1}(z) - \dt^{-1}(z) | &= |\dt^{-1}\circ \dt \circ \ft^{-1}(z) - \dt^{-1}(z) |\\
&\leq L| \dt(\ft^{-1}(z)) - z | \\
&= L|\dt(\ft^{-1}(z)) - \ft(\ft^{-1}(z)) | \\
&\leq LS_1 e^{\alpha \Re \ft^{-1}(z) } \\
& \leq LS_1e^{\alpha( \Re(z) + \log(1/\lambda) ) }\\
&\leq LS_1(1/\lambda)^{\alpha} e^{\alpha \Re(z)}
\end{align*}
This is \eqref{eq:thm1eq1} of Theorem \ref{thm:1} with $T_1 = LS_1(1/\lambda)^{\alpha}$.

\subsection{Condition(b)} Suppose that 
\begin{equation}
\label{eq:rep2} 
|\dt_z (u) - \dt_z(v) | \leq S_2 |u-v|^{\beta}
\end{equation}
for $\Re(u),\Re(v) < \log R$. 

\begin{lemma}
\label{lem:rep1}
With the hypotheses above, we have
\[ |\mu_{\dt}(u) - \mu_{\dt}(v)| \leq 2L^3S_2|u-v|^{\beta}.\]
\end{lemma}

\begin{proof}
Using the fact that $\dt$ is $L$-bi-Lipschitz with Lemma \ref{lem:0}, the triangle inequality and \eqref{eq:rep2}, we have
\begin{align*}
|\mu_{\dt}(u) - \mu_{\dt}(v) | &= \left | \frac{ \dt_{\zbar}(u)}{\dt_z(u)} - \frac{ \dt_{\zbar}(v) }{\dt_z(v)} \right | \\
&= \frac{ | \dt_{\zbar}(u) \dt_z(v) - \dt_{\zbar}(v)\dt_z(u) | }{ | \dt_{z}(u)| \cdot |\dt_z(v)| } \\
&\leq L^2 \left ( | \dt_{\zbar}(u) \dt_z(v) - \dt_{\zbar}(u)\dt_z(u) | + | \dt_{\zbar}(u) \dt_z(u) - \dt_{\zbar}(v)\dt_z(u) | \right )\\
&\leq 2L^3S_2|u-v|^{\beta},
\end{align*}
as required.
\end{proof}

Recall that the Jacobian is given by $J_f(u) = |f_z(u)|^2 - |f_{\zbar}(u)|^2$.

\begin{lemma}
\label{lem:rep2}
With the hypotheses above, we have
\[ |J_{\dt}(u) - J_{\dt}(v)| \leq 4LS_2|u-v|^{\beta}.\]
\end{lemma}

\begin{proof}
Using Lemma \ref{lem:0} with the inequality $|\dt_{\zbar}| \leq |\dt_z|$, the triangle inequality and \eqref{eq:rep2}, we have
\begin{align*}
|J_{\dt}(u) - J_{\dt}(v) | &= \left | |\dt_z(u)|^2 - |\dt_{\zbar}(u)|^2 - |\dt_z(v)|^2 + |\dt_{\zbar}(v)|^2 \right | \\
&\leq \left | ( |\dt_z(u)| + |\dt_{\zbar}(u)| )( |\dt_z(u)| - |\dt_z(u) | ) \right | + \left | ( |\dt_z(v)| + |\dt_{\zbar}(v)| )( |\dt_z(v)| - |\dt_z(v) | ) \right | \\
&\leq 4LS_2|u-v|^{\beta},
\end{align*}
as required.
\end{proof}

Now, using the formula for the complex dilatation of an inverse, Lemma \ref{lem:rep1}, Lemma \ref{lem:rep2}, and the inequality
\begin{equation}
\label{eq:rep2a} 
J_{\dt}(u) = |\dt_z(u) | \cdot ( 1 - |\mu_{\dt}(u)|^2) \geq \frac{ 1- ||\mu_{\dt}||_{\infty}^2 }{L^2},
\end{equation}
we have
\begin{align*}
| (\dt^{-1})_z(\dt (u) ) &- (\dt^{-1})_z(\dt(v)) | = \left | \frac{ \overline{ \dt_z(u) }}{J_{\dt}(u) } - \frac{ \overline{ \dt_z(v) }}{J_{\dt}(v) } \right | \\
&= \frac{ |\overline{ \dt_z(u) } J_{\dt}(v) - \overline{\dt_z(v)} J_{\dt}(u) | }{J_{\dt}(u) J_{\dt}(v) } \\
& \leq \left ( \frac{L^2}{1-||\mu_{\dt}||_{\infty}^2 } \right )^2 \cdot \left ( |\overline{ \dt_z(u) } J_{\dt}(v)  - \overline{ \dt_z(u) } J_{\dt}(u) | 
+ |\overline{ \dt_z(u) } J_{\dt}(u)  - \overline{ \dt_z(v) } J_{\dt}(u) | \right )\\
&\leq  \left ( \frac{L^2}{1-||\mu_{\dt}||_{\infty}^2 } \right )^2 \cdot \left ( L | J_{\dt}(u) - J_{\dt}(v) | + L^2(1-||\mu_{\dt}||_{\infty}^2 )| \dt_z(u) - \dt_z(v) | \right ) \\
&\leq S_2' |u-v|^{\beta},
\end{align*}
where $S_2'$ depends on $S_2, L$ and $||\mu_{\dt}||_{\infty}$.
Finally, we have
\[ | (\dt^{-1})_z(u) - (\dt^{-1})_z(v)| \leq S_2' |\dt(u) - \dt(v) |^{\beta} \leq S_2' L^{\beta} |u-v|^{\beta},\]
which is the first part of \eqref{eq:thm1eq2} with $T_2 = S_2' L^{\beta}$.
For the other part of \eqref{eq:thm1eq2}, we write
\[ | (\dt^{-1})_{\zbar}(\dt(u)) - (\dt^{-1})_{\zbar}(\dt(v))| = \left | -\frac{ \dt_{\zbar}(u)}{J_{\dt}(u) } + \frac{\dt_{\zbar}(v)}{J_{\dt}(v) } \right | \]
and proceed again using Lemma \ref{lem:rep1}, Lemma \ref{lem:rep2} and the assumption that $|\dt_{\zbar}(u) - \dt_{\zbar}(v)| \leq S_2 |u-v|^{\beta}$. We omit these computations as they are similar to those above.

\subsection{Condition (c)} Suppose that
\begin{equation}
\label{eq:rep3a} 
|\ft_z(u) - \dt_z(u) | \leq S_3 e^{\beta' \Re(u) },\quad |\ft_{\zbar}(u) - \dt_{\zbar}(u) | \leq S_3 e^{\beta ' \Re(u) }
\end{equation}
for $\Re (u) < \log R$. Given $\delta>0$, we may assume that $R$ is chosen small enough that $S_3R^{\beta '} < \delta$ and hence
\begin{equation}
\label{eq:rep3}
L-\delta \leq |\ft_z(u)| \leq L + \delta
\end{equation}
for $\Re(u) < \log R$. We will then also use the fact that
\begin{equation}
\label{eq:rep4}
J_{\ft}(u) \leq |\ft_z(u)|^2 < (L+ \delta)^2.
\end{equation}
Additionally, we have a lower bound for $J_{\ft}(u)$ by using \eqref{eq:rep3} to obtain
\begin{align}
\label{eq:rep5}
J_{\ft}(u) &= |\ft_z(u)|^2 - |\ft_{\zbar}(u)|^2\\
&\nonumber \geq (|\dt_z(u)| - \delta)^2 - (|\dt_{\zbar}(u)| + \delta)^2\\
&\nonumber \geq J_{\dt}(u) - 2L\delta.
\end{align}
For our final preparatory step, we have by the triangle inequality, \eqref{eq:rep3a} and \eqref{eq:rep3} that
\begin{align}
\label{eq:rep5a}
|J_{\dt}(u) - J_{\ft}(u)| &= \left | |\dt_z(u)|^2 - |\dt_{\zbar}(u)|^2 - |\ft_z(u)|^2 + |\ft_{\zbar}(u)|^2 \right | \\
&\nonumber \leq (|\dt_z(u)| + |\ft_z(u)| ) \cdot \left | | \dt_z(u) | - |\ft_z(u)| \right | + ( |\dt_{\zbar}(u)| + |\ft_{\zbar}(u)| ) \cdot \left |  |\dt_{\zbar}(u)| - |\ft_{\zbar}(u)| \right |\\
&\nonumber \leq 2(2L+\delta) S_3 e^{\beta' \Re(u)}. 
\end{align}

Now, writing $p = \ft^{-1}(u)$ and $q = \dt^{-1}(u)$, we have by the formula for the complex derivative of an inverse,
\begin{equation}
\label{eq:rep6}
| (\ft^{-1})_z (u) - (\dt^{-1})_z(u)| = \left | \frac{ \overline{ \ft_z(p) } }{J_{\ft} (p) } - \frac{ \overline{ \dt_z(q) } }{J_{\dt} ( q) } \right | 
= \frac{ | \overline{\ft_z(p)}J_{\dt}(q) - \overline{\dt_z(q)} J_{\ft}(p) | }{J_{\ft}(p) J_{\dt}(q) } .
\end{equation}
By \eqref{eq:rep2a} and \eqref{eq:rep5} we have
\[ J_{\ft}(p) J_{\dt}(q) \geq J_{\dt}(q)(J_{\dt}(p) - 2L\delta).\]
By choosing $\delta>0$ small enough, we may assume that
\begin{equation}
\label{eq:rep7}
J_{\ft}(p) J_{\dt}(q) \geq \frac{ 1 - ||\mu_{\dt}||_{\infty}^2 }{2L^2}.
\end{equation}
Next, a repeated application of the triangle inequality yields
\begin{align*}
| \overline{\ft_z(p)}J_{\dt}(q) - \overline{\dt_z(q)} J_{\ft}(p) | &\leq |\overline{\ft_z(p)} J_{\dt}(q) - \overline{\ft_z(p)} J_{\dt}(p) | +| \overline{\ft_z(p)} J_{\dt}(p) - \overline{\ft_z(p)}J_{\ft}(p) | \\
&\quad + |\overline{\ft_z(p)}J_{\ft}(p) - \overline{\dt_z(p)} J_{\ft}(p) | + | \overline{\dt_z(p)} J_{\ft}(p)  - \overline{\dt_z(q)} J_{\ft}(p) |\\
&= |\ft_z(p)| \cdot \left ( | J_{\dt}(q) - J_{\dt}(p)| + |J_{\dt}(p) - J_{\ft}(p) | \right ) \\ &\quad+ J_{\ft}(p) \cdot \left ( | \ft_z(p) - \dt_z(p)| + |\dt_z(p) - \dt_z(q)| \right).
\end{align*}
By \eqref{eq:rep2}, \eqref{eq:rep3a}, \eqref{eq:rep3}, \eqref{eq:rep4}, \eqref{eq:rep5a} and Lemma \ref{lem:rep2}, we have
\begin{align*}
| \overline{\ft_z(p)}J_{\dt}(q) - \overline{\dt_z(q)} J_{\ft}(p) | &\leq (L+\delta) \left ( 4LS_2 |p-q|^{\beta} + 2(2L+\delta) S_3 e^{\beta' \Re(p) } \right )\\
& \quad + (L+\delta)^2 \left ( S_3 e^{\beta' \Re(p) } + S_2 |p-q|^{\beta} \right ).
\end{align*}
To deal with the $|p-q|^{\beta}$ terms, by \eqref{eq:thm2eq2}, we have
\[ |p-q|^{\beta} = |\ft^{-1}(u) - \dt^{-1}(u)|^{\beta} \leq T_1^{\beta} e^{\alpha \beta \Re(u)}.\]
Putting this together with \eqref{eq:rep6} and \eqref{eq:rep7}, we conclude that there exists a constant $T_3$ depending only on the data from the hypotheses on $\ft$ and $\dt$ such that
\[ | (\ft^{-1})_z (u) - (\dt^{-1})_z(u)| \leq T_3 e^{\nu \Re(u)},\]
where $\nu = \min \{ \beta', \alpha \beta \}$ as before. This gives the first part of \eqref{eq:thm2eq3}. For the other part of \eqref{eq:thm2eq3}, we start with
\[ | (\ft^{-1})_{\zbar} (u) - (\dt^{-1})_{\zbar}(u) | = \left |- \frac{\ft_{\zbar}(p) }{J_{\ft}(p) } + \frac{ \dt_{\zbar}(q) }{J_{\dt}(q)} \right | \]
and proceed as above to conclude that
\[  | (\ft^{-1})_{\zbar} (u) - (\dt^{-1})_{\zbar}(u) | \leq T_3 e^{\nu \Re(u)}.\]
We again omit the computations as they are similar to the above.

\subsection{Completing the proof} 

\begin{proof}[Proof of Theorem \ref{thm:2}]
Assume first that $f$ has a geometrically repelling fixed point at $0$ with parameter $\lambda$ satisfying \eqref{eq:thm2eq4}.
As the hypotheses of Theorem \ref{thm:2} show that $\ft^{-1}$ and $\dt^{-1}$ satisfy the hypotheses of Theorem \ref{thm:1} with $1/\lambda$ satisfying \eqref{eq:thm1eq4},
there is an asymptotically conformal map $\widetilde{\psi}$ such that
\[ \widetilde{\psi} \circ \ft^{-1} = \dt^{-1} \circ \widetilde{\psi}.\]
As we can write this equation as $\dt \circ \widetilde{\psi} = \widetilde{\psi} \circ \ft$ and as $\widetilde{\psi}(z+2\pi i ) = \widetilde{\psi}(z) + 2\pi i$, we conclude that there is an asymptotically conformal map $\psi$ such that 
\[ D \circ \psi = \psi \circ f\]
as required.

If $0$ is a superrepelling fixed point of $f$, then we find a neighbourhood $|z|<R'$ of $0$ on which $|f(z)| > \lambda |z|$ holds for $\lambda$ satisfying \eqref{eq:thm2eq4} and apply the argument above.
\end{proof}

\section{Recovering the classical results}
\label{sec:old}

We first recover K\"onig's Theorem

\begin{proof}[Proof of Proposition \ref{prop:holo}]
As $f$ is holomorphic and fixes $0$, $f$ is simple at $0$ with the only generalized derivative being $z \mapsto e^{i\arg f'(0)} z$. Therefore the asymptotic representative is $\mathcal{D}(z) = \rho_f(|z|) e^{i\arg f'(0)} z/|z|$. Now,
\[ | f(B(0,r))| = \int_{B(0,r)} J_f = \int_0^{2\pi} \int_0^r |f'(te^{i\theta}) | ^2 t \: dt \: d\theta.\]
Writing $f(z) = a_1z +a_2z^2 + \ldots$, we have
\[ |f'(te^{i\theta}) |^2 = \left ( a_1 + 2a_2 te^{i\theta} + \ldots \right ) \cdot \left ( \overline{a_1} + 2\overline{a_2} te^{-i\theta} + \ldots \right) .\]
As $\int_0^{2\pi} e^{mi\theta} \: d\theta = 0 $ for $m\in \Z \setminus \{ 0 \}$, it follows that
\begin{align*}
| f(B(0,r))| &= 2\pi \int_0^r \left ( |a_1|^2t + 4|a_2|^2 t^3 + 9|a_3|^2 t^5 + \ldots \right ) \: dt \\
&= \pi \sum_{n=1}^{\infty} n |a_n|^2 r^{2n}.
\end{align*}
Hence
\[ \rho_f(r) = \left ( \sum_{n=1}^{\infty} n |a_n|^2 r^{2n} \right )^{1/2} =|a_1|r(1 + O(r^2))\]
as $r\to 0$. It then follows from Sternberg's Linearization Theorem, Theorem \ref{thm:ste}, that $\rho_f(r)$ may be smoothly conjugated in a one-sided neighbourhood of $0$ to the map $z\mapsto |a_1 |z$.

More precisely, suppose that $A(r) = |a_1|r$ and that $h$ is the conjugating map so that $h\circ \rho_f = A\circ h$ holds on $[0,r_0]$ for some $r_0 >0$. Define $H(re^{i\theta}) = h(r)e^{i\theta}$. Then, recalling that $a_1 = f'(0)$,
\begin{align*} H(\mathcal{D}(z)) &= H\left(  \rho_f(|z|) e^{i\arg f'(0)} z/|z| \right) \\
	&= h( \rho_f(|z|) ) e^{i\arg f'(0)} z/|z| \\
	&= A ( h(|z|)) e^{i\arg f'(0)} z/|z| \\
	&= |f'(0)| h(|z|) e^{i\arg f'(0)} z/|z|\\
	&= f'(0) H(z)
\end{align*}
from which we see that $H$ conjugates $\mathcal{D}$ to $z\mapsto f'(0)z$ in a neighbourhood of $0$. 
\end{proof}

Next, we recover B\"ottcher's Theorem.

\begin{proof}[Proof of Proposition \ref{prop:holo2}]
The only generalized derivative of $f$ at $0$ is $g(z) = e^{i\arg a_d} z^d$ and hence the asymptotic representative is $\mathcal{D}(z) = \rho_f(|z|) e^{i\arg a_d} z^d / |z|^d$.

Following similar computations as the proof of Proposition \ref{prop:holo}, we have
\[ \int_{B(0,r)} J_f =  \pi \sum_{n=d}^{\infty} n|a_n|^2 r^{2n} .\]
Now as $r\to 0$, $f$ gets closer and closer to a $d$-fold covering onto the image of $B(0,r)$. Thus
\[ |f(B(0,r)) | = \frac{\pi}{d} \left ( d |a_d|^2 r^{2d} \right) (1+O(r)) \]
as $r\to 0$. We conclude that
\[ \rho_f(r) = |a_d| r^d ( 1+O(r))\]
as $r\to 0$.
We may conjugate this by a linear map to 
\[ A_1(r) = r \mapsto r^d(1+O(r))\]
as $r\to 0$.
The logarithmic transform of this map is
\begin{align*} 
\widetilde{A_1}(r) &= \log \left ( e^{dr}(1+O(e^{r} )) \right ) \\
&= dr + \log ( 1+ O(e^{r}) ) \\
&= dr + O(e^{r})
\end{align*}
as $r\to -\infty$.
Conjugating this by $r \mapsto -1/r$, we obtain
\[ \widetilde{A_2}(r) = -\frac{ 1}{-d/r + O(e^{-1/r}) } = \frac{r}{d} + O(e^{-1/r}) \]
as $r\to 0^+$. Applying Sternberg's Linearization Theorem, Theorem \ref{thm:ste}, we see that $\widetilde{A_2}(r)$ is conjugate to $\widetilde{A_3}(r) = r/d$. Undoing the logarithmic transform and the conjugation by $-1/r$, we conclude that $A_1$ is smoothly conjugate to $A(r) = r^d$. The rest of the proof now runs analogously to that of Proposition \ref{prop:holo}.

Suppose that $h\circ \rho_f = A \circ h$ holds on $[0,r_0]$ for some $r_0>0$. Define $H(re^{i\theta}) = h(r) e^{i\theta}$. Then
\begin{align*}
H(\mathcal{D}(z)) &= H \left ( \rho_f(|z|) e^{i\arg a_d} z^d / |z|^d \right )\\
&= h( \rho_f|z|) e^{i\arg a_d} z^d / |z|^d \\
&= A(h(|z|))e^{i\arg a_d} z^d / |z|^d\\
&= e^{i\arg a_d} H(z)^d.
\end{align*}
This shows that $H$ conjugates $\mathcal{D}$ to $e^{i\arg a_d} z^d$ and we may apply a final conjugation by a rotation to conclude that $\mathcal{D}$ is smoothly conjugate to $z^d$.
\end{proof}

\section{Uniqueness of the conjugacy}
\label{sec:conj}

As noted in the introduction, the uniqueness of the conjugacy in Theorem \ref{thm:1} and Theorem \ref{thm:2} boils down to finding which maps conjugate $\dt$ to itself.

As $\mathcal{D}(z) = \rho_f(|z|)g(z/|z|)$, where $g :S^1 \to \C$ is BLD and $\rho_f$ is the mean radius function, we have
\[ \dt(x+iy) = g_1(y) + \widetilde{\rho}_f(x) + ig_2(y),\]
c.f. \cite[Section 7]{FP} in dimension two. It's worth bearing in mind that if $\mathcal{D}(z) = z^d$ then $g_1(y) \equiv 0$, $g_2(y) = dy$ and $\widetilde{\rho}_f(x) = dx$.

Suppose then that we are looking to find solutions to $\dt \circ \eta = \eta \circ \dt$ and we work in the special case that $\eta(x+iy) = h(x) + iy$. Then 
\[ \dt(\eta(x+iy)) = g_1(y) + \widetilde{\rho}_f(h(x)) + ig_2(y)\]
and
\[ \eta ( \dt(x+iy)) = h( g_1(y) + \widetilde{\rho}_f(x) ) + ig_2(y).\]
We see that we are looking for functions which commute with $g_1(y) + \widetilde{\rho}_f(x)$, viewed as a function of $x$. In the continuous class, Lipinski's Theorem (see \cite[p.213]{Kuc}) states that such functions exist and depend only on an arbitrary continuous function. However, once such functions are required to be $C^1$ then the situation is much more restrictive and the only such functions are the regular iterates of the given function (see \cite[p.214]{Kuc}).

We leave open the general question of determining the quasiconformal conjugacies of $\dt$ to itself.

\end{document}